\documentclass[11pt, reqno]{amsart}%
\usepackage{amsmath,amsfonts,amssymb,amsthm}
\usepackage{epsfig}
\usepackage{tikz}

\usepackage{bbm}
\usepackage{enumerate}
\usepackage{amstext}
\usepackage{mathrsfs}
\usepackage{xspace}
\usepackage{amsfonts}
\usepackage{amsmath}
\usepackage{amssymb}
\usepackage{amstext}
\usepackage{amsthm}       
\usepackage{xspace}
\usepackage[active]{srcltx}
\usepackage{cite}

\renewcommand{\S}{\mathcal S}

\newcommand{\CC}{\D{C}}

\newcommand{\R}{\mathbb R}

\newcommand{\Lip}{\D{Lip}}
\newcommand{\N}{\mathbb N}

\newcommand{\TT}{\mathbb T}
\newcommand{\Z}{\mathbb Z}

\newcommand{\sol}{\mathfrak{S}}

\newcommand{\eps}{\varepsilon}
\newcommand{\ucv}{\rightrightarrows}

\newcommand{\cmeasures}{\tilde{\mathcal C}}
\newcommand{\MM}{{\R^d}}
\newcommand{\TM}{{\R^d\times\R^d}}
\newcommand{\TTorus}{{\TT^d\times\R^d}}

\newcommand{\omegasetG}{\omega\big(\{u^\lambda_G\}\big)}

\newcommand{\supp}{\operatorname{supp}}

\newcommand{\e}{\textrm{\rm e}}

\newcommand{\Mis}{\mathfrak{M}}

\newcommand{\tagliato}{$\kern-5 mm -$}
\newcommand{\tagliat}{$\kern-4 mm -$}

\newcommand{\D}[1]{\mbox{\rm #1}}

\newcommand{\weakst}{\stackrel{\ast}{\rightharpoonup}}
\newcommand{\weakcv}{\rightharpoonup}

\newtheorem{teorema}{Theorem}[section]
\newtheorem{prop}[teorema]{Proposition}
\newtheorem{lemma}[teorema]{Lemma}
\newtheorem{definition}[teorema]{Definition}
\newtheorem{cor}[teorema]{Corollary}

\newtheorem{guess}[teorema]{Remark}
\newtheorem{example}[teorema]{Example}

\newenvironment{esempio}{\begin{example} \begin{rm}}{\end{rm} \end{example}}

\newenvironment{oss}{\begin{guess} \begin{rm}}{\end{rm} \end{guess}}

\newcommand{\IN}{\text{ in }}

\newcommand{\FORALL}{\text{ for all }}

\usepackage{hyperref}

\usepackage{xcolor}

\begin{document}
\title{Vanishing discount selection for constant-sign local perturbations of periodic Hamiltonians}
\author{Andrea Davini}
\address{Dip. di Matematica, {Sapienza} Universit\`a di Roma,
P.le Aldo Moro 2, 00185 Roma, Italy}
\email{davini@mat.uniroma1.it}
\date{\today}

\begin{abstract}
We study the vanishing discount selection problem for critical
Hamil\-ton--\allowbreak Jacobi equations posed on $\R^d$ with Hamiltonian $G=H-V$,
where $H$ is continuous, periodic in space, convex and superlinear
in the momentum, and $V$ is a compactly supported potential of constant sign. 
For $V\leqslant0$, we prove local uniform convergence of the
discounted solutions to a distinguished critical solution,
which we characterize in terms of Mather measures and of the
solution selected by the unperturbed periodic problem.
The proof relies on a decomposition of limiting discounted measures
into a retained component and a periodic component accounting for the 
mass lost at infinity.
For $V\geqslant0$, we establish convergence under the additional
assumption that the unperturbed selected solution has zero average
with respect to every periodic Mather measure. We also give explicit examples in which the discounted solutions
converge, but natural formulas for the candidate limit fail
to identify the selected solution.
\end{abstract}
\maketitle

\section*{Introduction}\label{sec:intro}

We study the vanishing discount selection problem for critical
Hamilton--Jacobi equations on $\R^d$ with Hamiltonian $G=H-V$,
where $H$ is continuous, periodic in space, convex and superlinear
in the momentum, and $V$ is a continuous, compactly supported potential.
For each $\lambda>0$, let $u_G^\lambda$ be the unique bounded
viscosity solution of
\[
 \lambda u_G^\lambda+G(x,Du_G^\lambda)=c(G)
 \qquad\hbox{in $\MM$},
\]
where $c(G)$ is the critical value associated with bounded subsolutions;
see Section~\ref{subsec:critical}.
When $V\equiv0$, $c(G)$ coincides with the standard periodic
critical value $c(H)$, also known as ergodic constant.
The question is whether $u_G^\lambda$ converges as $\lambda\to0^+$
and which solution of $G(x,Du)=c(G)$ in $\MM$ is selected, since the critical
equation generally has many solutions, even up to additive constants.

The one-dimensional version of this problem was studied in
\cite{CDD23}, where convergence and a characterization of the
selected limit were established for compactly supported potentials
of arbitrary sign. The present paper pursues this program in higher
dimensions. A local perturbation destroys the compactness provided by
the reduction to the torus, while leaving the Hamiltonian
unchanged outside a compact set. Selection must therefore account for both
the perturbation and the minimizing dynamics that persist at infinity.

On compact manifolds, convergence under general continuity,
convexity and coercivity assumptions was established in
\cite{DFIZ1}, where the selected solution is characterized through
inequalities against Mather measures. Earlier results under
additional dynamical assumptions were obtained in \cite{IS11}.

The subsequent literature has developed this selection approach
in several directions, including degenerate elliptic equations
\cite{MiTra17,IMTa17}, boundary value problems \cite{IMTb17},
weakly coupled and monotone systems \cite{DZ21,Is21,IsLi20},
and nonlocal equations \cite{DI27}.
Related selection problems also arise in mean field games
\cite{CaPo19}.
Other developments concern nonlinear and degenerate discounting
\cite{CCIsZh19,Chen21,Z20,CFZZ24}, negative or nonmonotone
discounting \cite{DaLi21,DNYZ24}, and certain classes of
nonconvex Hamiltonians \cite{GMT18}.
In many of these settings, suitable notions of minimizing measures
remain central to the identification of the limit.

In the whole space, convergence was established in
\cite[Theorem~1]{BK22} for mechanical Hamiltonians of the form
$\frac12|p|^2-f(x)$, where $f$ is bounded, continuous and attains
its minimum. Neither periodicity of $f$ nor compactness of its set of
global minimizers is required. The convergence proof exploits the eikonal
structure and uniqueness with prescribed values on the
minimizing set.

An earlier convergence result for general convex Hamiltonians
in the whole space was obtained in \cite{IsSic20}
under a condition ensuring tightness of the families of measures
used in the asymptotic analysis; see also \cite{TZ26} for an
extension to nonlinear discounting.
This localization mechanism is closely related to the one at work
when $c(G)>c(H)$ in \cite{CDD23}: a bounded critical subsolution
that is uniformly strict outside a compact set prevents loss of
mass at infinity.
As already discussed in \cite{CDD23}, the two settings have a
nonempty intersection, and the proof for the positive-gap case
in that paper already applies in arbitrary dimension,
\emph{mutatis mutandis}.

The main difficulty arises when $c(G)=c(H)$,
since the measures may then lose part or all of their mass
at infinity.
In \cite{CDD23}, this loss is handled by exploiting the geometry
of one-dimensional minimizing curves.\smallskip

Our main result establishes convergence for every nonpositive
compactly supported potential in any dimension
(Theorem~\ref{thm:main}). More precisely, if $V\leqslant0$, then
\begin{equation}\label{eq:intro-convergence}
 u_G^\lambda\ucv u_G^0\quad\hbox{in $\MM$ as $\lambda\to0^+$},
\end{equation}
where the convergence is locally uniform and $u_G^0$ is a Lipschitz
critical solution bounded from below. The limit is the supremum
of bounded critical subsolutions with nonpositive averages against
all whole-space Mather measures for $G$; when $c(G)=c(H)$, the
candidates must also lie below $u_H^0$, where the latter denotes the solution 
selected by the vanishing
discount approximation for the unperturbed periodic Hamiltonian
$H$; see \cite{DFIZ1}.
The measures and the precise constraints are specified in
Sections~\ref{subsec:closed} and \ref{sec:selection}.
The limit is bounded when the critical values agree and coercive  
when $c(G)>c(H)$; see Proposition~\ref{prop:gap}.

As in \cite{DFIZ1}, we establish the convergence in
\eqref{eq:intro-convergence} by studying the asymptotic behavior
of a family of probability measures, called discounted measures.
More precisely, for each fixed $x\in\MM$ and $\lambda>0$,
we associate with any optimal trajectory for $u_G^\lambda(x)$
an exponentially weighted occupation measure
$\tilde\mu_x^\lambda$ on $\TM$. 
This measure records the time spent by the trajectory in phase
space, and the weight is normalized to give
total mass one.

The challenging step in the noncompact setting considered here 
is to recover the escaping mass through periodicity.
In \cite{CDD23}, this issue is addressed by fixing a sufficiently
large ball centered at the origin and containing both $x$ and
the support of $V$, and splitting $\tilde\mu_x^\lambda$ into
two components corresponding to the weighted time spent inside
and outside the ball.
The outer component is projected onto the torus.
The crucial point is to show that the limiting measures obtained
from these two components as $\lambda\to0^+$ are closed.
After normalization, the nonzero limits are then identified
as Mather measures for $G$ on $\R$ and for the periodic
problem associated with $H$, respectively.

Proving closedness of the limiting measures requires controlling
the boundary terms generated whenever the trajectory enters
or leaves the ball.
In one dimension, the optimal trajectory crosses the ball
at most once.
The underlying reason is topological: roughly speaking, optimality
prescribes the velocity at each point, whereas leaving the ball
and subsequently returning would require crossing the same
boundary point with velocities of opposite signs.

In higher dimensions, an exit and a subsequent return may occur
at different boundary points, so this obstruction disappears. 
A bound on the number of crossings growing more slowly than
$1/\lambda$ would suffice to kill the boundary contributions
in the limit.
Although our intuition suggests that such a bound should hold,
a proof has so far eluded us.

The new idea used here is to separate the retained and escaping
mass after passing to the limit.
We take limits of the full discounted measures on the tangent
bundle of $\R^d$ and of their projections onto the tangent
bundle of $\TT^d$, along a common subsequence.
The first limit records the retained mass, while the second
preserves the total mass.

We then subtract the projection of the first limit onto the
tangent bundle of the torus from the second.
The resulting residual measure is a nonnegative closed measure
on this bundle whose mass is exactly the mass lost at infinity. 
This construction requires no control on the number of
boundary crossings and thus avoids the geometric difficulty
described above.

An action identity analogous to that used in \cite{CDD23}
then shows that, after normalization, the retained component
minimizes the action for $G$, while the residual measure
minimizes the periodic action for $H$.
Each nonzero component therefore yields a Mather probability
measure. 
The same identity excludes loss of mass when $c(G)>c(H)$.
When the critical values agree, the escaping component can
survive and explains why the periodic selected solution enters
the characterization of the limit.
The decomposition itself holds without a sign assumption on $V$;
see Proposition~\ref{prop:decomp}.

For nonnegative potentials, we prove the existence of bounded
critical solutions and uniform bounds for the discounted family.
We also obtain convergence to the least bounded critical solution
above $u_H^0$, provided that $u_H^0$ has zero average against every
periodic Mather measure (Theorem~\ref{thm:positive}).
This assumption is restrictive: it fails even in the
one-dimensional asymmetric double-pendulum example presented
in Section~\ref{subsec:1d}. The theorem therefore does not settle the general
multidimensional convergence problem for nonnegative potentials
under the standing assumptions.

Two explicit examples presented in Section~\ref{sec:examples}
illustrate the difficulty of identifying the right formula
for the candidate limit in the case of nonnegative potentials.
In each example, the discounted solutions converge, but the
proposed formula yields a candidate different from their
actual limit.

These examples suggest that a general selection principle must
also account for information carried by minimizing trajectories
that escape to infinity.
Such information already enters the one-dimensional theory:
the characterization in \cite[Theorem~2(II)]{CDD23} involves
a comparison with the periodic selected solution in a direction
determined by the minimizing dynamics.
In higher dimensions, the residual measures on the torus
retain information about escaping trajectories that the
whole-space Mather measures alone cannot capture.
This suggests a role for both families of measures in the
search for a more general selection principle.\smallskip

The paper is organized as follows.
Section~\ref{sec:prelim} introduces the notation and standing
assumptions and collects preliminary results on critical values,
closed measures and discounted equations.
Section~\ref{sec:bounds} establishes bounds and compactness
properties for the discounted family when the potential has
constant sign, as well as the existence of bounded critical
solutions for nonnegative potentials.
Section~\ref{sec:measures} introduces discounted measures,
derives the integral inequalities used in the proofs and
develops the decomposition accounting for loss of mass at infinity.
Section~\ref{sec:selection} proves the selection theorem for
nonpositive potentials and derives further properties of the
selected solution and Mather measures.
Section~\ref{sec:positive} proves the selection result for
nonnegative potentials under the additional condition on
periodic Mather averages.
Finally, Section~\ref{sec:examples} presents the counterexamples
and discusses the role of information from infinity in the
search for a more general selection principle.

\section{Preliminaries}\label{sec:prelim}
\numberwithin{equation}{section}
\setcounter{equation}{0}

\subsection{Notations and standing assumptions}

The scalar product in $\R^d$ will be denoted by
$\langle\,\cdot\,,\cdot\,\rangle$, while $|\cdot|$ stands for the
Euclidean norm. We denote by $B_r(x)$ the open ball of radius $r$ centered
at $x$, and write $B_r$ when $x=0$.
The flat torus $\R^d/\Z^d$ is denoted by $\TT^d$.
Functions on $\TT^d$ are identified with their $\Z^d$--periodic
lifts to $\R^d$.
Throughout the paper, $M$ denotes either $\MM$ or $\TT^d$,
according to the context.

We denote by $\CC(M)$ the space of continuous real-valued functions
on $M$, and by $\CC_b(M)$ and $\CC_c(M)$ its subspaces of bounded
functions and functions with compact support, respectively.
For $k=1,2,\ldots$, the space $\CC^k(M)$ consists of functions with
continuous derivatives up to order $k$, while $\CC^\infty(M)$ denotes
the space of smooth functions. We set
\[
 \CC_b^k(M):=\CC^k(M)\cap\CC_b(M),\qquad
 \CC_c^k(M):=\CC^k(M)\cap\CC_c(M),
 \qquad k=1,2,\ldots,\infty.
\]
Thus the subscript $b$ imposes boundedness on the function itself.
The space of globally Lipschitz real-valued functions on $M$ is
denoted by $\Lip(M)$. We use the same notation on open subsets
and tangent bundles. We write $u_n\ucv u$ in $M$ to mean that $u_n$
converges to $u$ uniformly on compact subsets of $M$.
The symbol $\|\cdot\|_\infty$ denotes the supremum norm, or the
essential supremum norm when appropriate.

By a curve we mean an absolutely continuous map from an interval
into $M$, with absolute continuity understood locally on unbounded
intervals. The corresponding class on $[a,b]$ is denoted by
$\mathrm{AC}([a,b];M)$. We write $\Lip_{y,x}([a,b];M)$ for the
Lipschitz curves $\gamma$ satisfying $\gamma(a)=y$ and $\gamma(b)=x$.

Throughout the paper, $H\in\CC(\TM)$ satisfies the following
assumptions:
\begin{itemize}
\item[(H1)] for every $x\in\MM$, the map $p\mapsto H(x,p)$ is convex;
\smallskip
\item[(H2)] there exist superlinear functions $\alpha,\beta:\R_+\to\R$
such that
\[
 \alpha(|p|)\leqslant H(x,p)\leqslant\beta(|p|)
 \qquad\hbox{for every $(x,p)\in\TM$;}
\]
\item[(H3)] $H(x+z,p)=H(x,p)$ for every $(x,p)\in\TM$ and $z\in\Z^d$.
\end{itemize}
Here superlinearity means that $\alpha(r)/r$ and $\beta(r)/r$ tend to
$+\infty$ as $r\to+\infty$. Given $V\in\CC_c(\MM)$, we set
\begin{equation*}
 G(x,p):=H(x,p)-V(x)\qquad\hbox{for every $(x,p)\in\TM$.}
\end{equation*}
The Hamiltonian $G$ satisfies (H1)--(H2), with possibly different
functions $\alpha$ and $\beta$. Both Hamiltonians
are uniformly continuous on $\MM\times\overline B_R$ for each $R>0$.
Convexity also implies a uniform Lipschitz bound in the momentum on each
such set.

Solutions, subsolutions and supersolutions will always be continuous and
understood in the viscosity sense; see \cite{barles,bardi}.
We recall that, for a convex
Hamiltonian $F$ satisfying the preceding uniform assumptions, a
Lipschitz function is a subsolution of $F(x,Du)=a$ if and only if
$F(x,Du(x))\leqslant a$ almost everywhere; see \cite{barles}.
Coercivity implies that all
subsolutions at a fixed level are equi--Lipschitz. In particular, a
possible common Lipschitz constant is
\begin{equation*}
 \kappa_a(F):=\sup\{|p|:F(x,p)\leqslant a\hbox{ for some }x\in\MM\}.
\end{equation*}
The same almost-everywhere characterization holds for
$\lambda u+F(x,Du)=a$. Every bounded subsolution of this equation is
globally Lipschitz; see \cite[Lemma~2.5]{barles}.

We will repeatedly use that maxima, minima and convex combinations of
two critical subsolutions are still subsolutions. For minima this uses
convexity; it can also be checked directly from the almost-everywhere
characterization. The same statement holds for a pointwise supremum or
infimum of equi--Lipschitz subsolutions whenever the resulting function
is finite; see \cite[Proposition~2.1]{DFIZ1}.

The following standard smooth approximation lemma will be useful in
integrating subsolution inequalities against measures.
For regularization arguments of this type, see \cite{barles};
more general geometric formulations are given in
\cite[Theorems~8.1 and~8.5]{fatmad}.
We include the proof under our standing assumptions.

\begin{lemma}\label{lem:smoothing}
Let $F$ be either $H$ or $G$, and let $w\in\Lip(\MM)$ satisfy
\[
 F(x,Dw)\leqslant a(x)\qquad\hbox{in $\MM$,}
\]
where $a$ is bounded and uniformly continuous. For every $\eps>0$ there
exists $w_\eps\in\CC^\infty(\MM)\cap\Lip(\MM)$ such that
\begin{equation*}
 \|w_\eps-w\|_\infty\leqslant\eps,
 \qquad F(x,Dw_\eps(x))\leqslant a(x)+\eps
 \quad\hbox{for every $x\in\MM$.}
\end{equation*}
The Lipschitz constant does not increase under this approximation.
If $w$ is bounded, then $w_\eps$ is bounded as well.
\end{lemma}

\begin{proof}
Let $\rho_\eta$ be a standard nonnegative convolution kernel supported
in $B_\eta$, and set $w_\eta=w*\rho_\eta$. If $\kappa$ is a Lipschitz
constant for $w$, then $\|w_\eta-w\|_\infty\leqslant\kappa\eta$ and
$\|Dw_\eta\|_\infty\leqslant\kappa$. By convexity,
\begin{align*}
 F(x,Dw_\eta(x))
 &\leqslant\int F(x,Dw(x-y))\rho_\eta(y)\,dy\\
 &\leqslant\int a(x-y)\rho_\eta(y)\,dy+\omega_F(\eta)
 \leqslant a(x)+\omega_a(\eta)+\omega_F(\eta),
\end{align*}
where $\omega_F$ is a modulus of continuity for $F$ on
$\MM\times\overline B_\kappa$ and $\omega_a$ is one for $a$.
Choosing $\eta$ sufficiently small proves the assertion.
\end{proof}

\subsection{Critical values and fundamental subsolutions}
\label{subsec:critical}

For $F=H,G$, we define the bounded critical value by
\begin{equation}\label{eq:cvalue}
 c(F):=\inf\big\{a\in\R:\hbox{$F(x,Du)=a$ in $\MM$\ \ admits bounded subsolutions}\big\}.
\end{equation}
The critical equation for the perturbed Hamiltonian is
\begin{equation}\tag{$\mathrm{HJ}_G$}\label{eq:criticalG}
 G(x,Du)=c(G)\qquad\hbox{in $\MM$.}
\end{equation}
We also introduce the free critical value
\begin{equation*}
 c_f(F):=\inf\big\{a\in\R:\hbox{$F(x,Du)=a$ in $\MM$\ \ admits subsolutions}\big\}.
\end{equation*}
The adjective free refers to the absence of any boundedness requirement.
We always have $c_f(F)\leqslant c(F)$. Both values are finite. Moreover,
the infimum defining $c_f(F)$ is attained: one takes subsolutions at
levels decreasing to $c_f(F)$, normalizes their value at the origin and
uses local compactness and viscosity stability. Thus subsolutions
exist at every level $a\geqslant c_f(F)$.

For the periodic Hamiltonian $H$, the value $c(H)$ agrees with the
periodic critical value. Indeed, if $w$ is a bounded subsolution at
level $a$, then
\[
 \inf_{z\in\Z^d}w(x+z),\qquad x\in\MM,
\]
is a finite periodic subsolution at the same level. The converse
inequality follows since every continuous periodic function is bounded.
In particular, by \cite{LPV}, there exists a periodic solution of
\begin{equation}\tag{$\mathrm{HJ}_H$}\label{eq:criticalH}
 H(x,Du)=c(H)\qquad\hbox{in $\MM$.}
\end{equation}

The next proposition compares the critical values of the perturbed
and unperturbed Hamiltonians.

\begin{prop}\label{prop:corder}
For every $V\in\CC_c(\MM)$, one has
\begin{equation*}
 c(G)\geqslant c(H).
\end{equation*}
If $V\geqslant0$, then $c(G)=c(H)$.
\end{prop}

\begin{proof}
Pick $a>c(G)$ and a bounded subsolution $w$ of $G(x,Dw)=a$.
Let $(z_n)_n\subset\Z^d$ satisfy $|z_n|\to+\infty$, and put
$w_n(x)=w(x+z_n)$. The functions $w_n$ are equi--bounded and
equi--Lipschitz, hence a subsequence converges locally uniformly to a
bounded function $v$. Since
\[
 G(x+z_n,p)=H(x,p)-V(x+z_n)\ucv H(x,p)\quad\hbox{in $\TM$,}
\]
viscosity stability gives $H(x,Dv)\leqslant a$. It follows that
$c(H)\leqslant a$. Letting $a\downarrow c(G)$ proves the first assertion.
When $V\geqslant0$, the inequality $G\leqslant H$ gives the reverse
inequality between critical values.
\end{proof}

The Lagrangian associated with $F$ is defined by
\begin{equation*}
 L_F(x,q):=\sup_{p\in\R^d}\big\{\langle p,q\rangle-F(x,p)\big\}.
\end{equation*}
It is continuous, convex in $q$, superlinear and bounded on bounded
velocity sets, uniformly with respect to $x$. In particular,
\begin{equation}\label{eq:LG}
 L_G(x,q)=L_H(x,q)+V(x)\qquad\hbox{for every $(x,q)\in\TM$.}
\end{equation}

For every $a\in\R$, define
\begin{equation}\label{eq:actionS}
 S_a^F(y,x):=\inf_{\substack{t>0\\
                 \gamma\in\Lip_{y,x}([-t,0];\MM)}}
       \int_{-t}^0\big(L_F(\gamma(s),\dot\gamma(s))+a\big)\,ds.
\end{equation}
For the classical action construction, see \cite[Section~2-1]{ConItu};
see also \cite{FS05} for the metric approach. We give the details
needed here under our standing assumptions.

\begin{prop}\label{prop:metric}
Let $a\in\R$. The following properties hold.
\begin{itemize}
\item[(i)] A function $w:\MM\to\R$ is a subsolution of
$F(x,Dw)=a$ if and only if
\begin{equation}\label{eq:domination}
 w(\gamma(t_1))-w(\gamma(t_0))
 \leqslant\int_{t_0}^{t_1}
 \big(L_F(\gamma(s),\dot\gamma(s))+a\big)\,ds
\end{equation}
for every Lipschitz curve $\gamma:[t_0,t_1]\to\MM$.
Any function satisfying this inequality is Lipschitz, and the
inequality is equivalent to
\[
 w(x)-w(y)\leqslant S_a^F(y,x)
 \qquad\hbox{for all $x,y\in\MM$.}
\]
\item[(ii)] If $a<c_f(F)$, then $S_a^F\equiv-\infty$.
\item[(iii)] If $a\geqslant c_f(F)$, then $S_a^F$ is finite and Lipschitz
in both variables, $S_a^F(y,y)=0$, and
\[
 S_a^F(y,x)\leqslant S_a^F(y,z)+S_a^F(z,x),\qquad
 S_a^F(y,x)\leqslant\kappa_a(F)|x-y|.
\]
For each $y\in\MM$, the function $S_a^F(y,\cdot)$ is the largest
subsolution at level $a$ vanishing at $y$.
\end{itemize}
\end{prop}

\begin{proof}
(i). Every subsolution satisfies \eqref{eq:domination} by
Lemma~\ref{lem:smoothing} and Fenchel's inequality.
Conversely, suppose that $w$ satisfies \eqref{eq:domination}.
Applying it to straight segments from $x$ to $y$ and from $y$ to $x$,
each traversed with speed one, gives
\[
 |w(x)-w(y)|\leqslant
 \left(\sup_{z\in\MM,\,|q|\leqslant1}|L_F(z,q)+a|\right)|x-y|.
\]
Thus $w$ is Lipschitz. At a differentiability point $x$, applying
\eqref{eq:domination} to $\gamma(s)=x+sq$, $0\leqslant s\leqslant h$,
dividing by $h$ and letting $h\to0^+$ yields
\[
 \langle Dw(x),q\rangle\leqslant L_F(x,q)+a
 \qquad\hbox{for every $q\in\R^d$.}
\]
Fenchel duality gives $F(x,Dw(x))\leqslant a$ almost everywhere,
hence also in the viscosity sense. The equivalence with the
increment inequality follows directly from \eqref{eq:actionS}.
\smallskip

(ii). Suppose that $a<c_f(F)$. A straight segment gives a finite
upper bound for $S_a^F(y,x)$. Concatenating with fixed connecting
segments shows that
if $S_a^F$ takes the value $-\infty$ at one pair, then it is
identically $-\infty$.
Otherwise, $S_a^F$ is finite everywhere, and concatenation gives
the triangle inequality. This and
\eqref{eq:actionS} imply that $S_a^F(y,\cdot)$ satisfies
\eqref{eq:domination}. By (i), it is a subsolution at level $a$,
contradicting $a<c_f(F)$. Thus $S_a^F\equiv-\infty$.
\smallskip

(iii). Suppose that $a\geqslant c_f(F)$. A subsolution $w$ at
level $a$ provides the finite lower bound
$S_a^F(y,x)\geqslant w(x)-w(y)$.
Together with the upper bound obtained from a straight segment,
this shows that $S_a^F$ is finite. Concatenation gives the triangle
inequality, and the argument in (ii) shows that $S_a^F(y,\cdot)$
is a subsolution at level $a$.

The lower bound gives $S_a^F(y,y)\geqslant0$,
while constant curves with duration tending to zero give the
reverse inequality. The increment inequality then shows that
$S_a^F(y,\cdot)$ is the largest subsolution vanishing at $y$.
Since subsolutions at level $a$ are $\kappa_a(F)$--Lipschitz,
we obtain $S_a^F(y,x)\leqslant\kappa_a(F)|x-y|$.
Together with the triangle inequality, this also gives
the Lipschitz bound in both variables.
\end{proof}

\subsection{Closed measures and the periodic selected solution}
\label{subsec:closed}

Let $TM$ denote the tangent bundle of $M$, identified with $M\times\R^d$.
We denote by $\pi_1:TM\to M$, $\pi_1(x,q)=x$, the canonical
projection onto the base. We denote measures on tangent bundles by
symbols with a tilde and their projections onto the base by the same
symbols without a tilde, for instance $\mu:=\pi_{1\#}\tilde\mu$.
This convention applies to both choices of $M$.
A nonnegative finite Borel measure $\tilde\mu$ on $TM$ is termed
\emph{closed} if
\begin{equation}\label{eq:closed}
 \int_{TM}|q|\,d\tilde\mu(x,q)<+\infty,
 \qquad
 \int_{TM}\langle D\varphi(x),q\rangle\,d\tilde\mu(x,q)=0
\end{equation}
for every $\varphi\in\CC_b^1(M)\cap\Lip(M)$. On the torus the test
space is simply $\CC^1(\TT^d)$. We include the zero measure among
closed finite measures, and denote by $\cmeasures(TM)$ the family of
closed probability measures. The latter is nonempty, since it contains
$\delta_{(x,0)}$ for every $x\in M$.

The following proposition gives two equivalent choices of the
test space on $\MM$.

\begin{prop}\label{prop:closed-tests}
Let $\tilde\mu$ be a nonnegative finite Borel measure on $\TM$ with
$\int_{\TM}|q|\,d\tilde\mu<+\infty$. The following conditions are
equivalent:
\begin{itemize}
\item[(i)] $\tilde\mu$ is closed;
\smallskip
\item[(ii)] the second identity in \eqref{eq:closed} holds for every
$\varphi\in\CC^1(\MM)\cap\Lip(\MM)$;
\smallskip
\item[(iii)] the second identity in \eqref{eq:closed} holds for every
$\varphi\in\CC_c^1(\MM)$.
\end{itemize}
\end{prop}

\begin{proof}[Proof of Proposition~\ref{prop:closed-tests}]
The implications (ii)$\Rightarrow$(i)$\Rightarrow$(iii) are immediate.
Assume (iii), and fix $\varphi\in\CC^1(\MM)\cap\Lip(\MM)$.
Subtracting a constant, we may suppose $\varphi(0)=0$, so that
$|\varphi(x)|\leqslant\ell|x|$, where $\ell$ is a Lipschitz
constant for $\varphi$.
Choose $\chi\in\CC_c^\infty(\MM)$ with $0\leqslant\chi\leqslant1$,
$\chi=1$ on $B_1$ and $\supp(\chi)\subset\overline B_2$.
For $n\geqslant1$, set $\chi_n(x):=\chi(x/n)$.
Applying (iii) to $\chi_n\varphi$ gives
\[
 0=\int_{\TM}\chi_n(x)\langle D\varphi(x),q\rangle\,d\tilde\mu
   +\int_{\TM}\varphi(x)\langle D\chi_n(x),q\rangle\,d\tilde\mu.
\]
The second integral tends to zero, since
\[
 \left|\int_{\TM}\varphi(x)\langle D\chi_n(x),q\rangle
       \,d\tilde\mu\right|
 \leqslant2\ell\|D\chi\|_\infty
       \int_{\{|x|\geqslant n\}}|q|\,d\tilde\mu
 \longrightarrow0.
\]
The integrands in the first integral converge pointwise to
$\langle D\varphi(x),q\rangle$ and are bounded in absolute value
by $\ell|q|$. The Dominated Convergence Theorem therefore proves
(ii).
\end{proof}

The distinction between finite measures and probability measures will
be useful when part of the discounted mass is lost at infinity. The
following elementary fact applies to both and does not require the
subsolution to be bounded.

\begin{lemma}\label{lem:action}
Suppose $F(x,Dw)\leqslant a(x)$ admits a subsolution on $M$, where
$a$ is bounded and uniformly continuous on $M$. Then
\begin{equation}\label{eq:action-pos}
 \int_{TM}\bigl(L_F(x,q)+a(x)\bigr)\,d\tilde\mu(x,q)\geqslant0
\end{equation}
for every nonnegative finite closed measure $\tilde\mu$ on $TM$.
\end{lemma}

\begin{proof}
Let $w$ be such a subsolution. Uniform coercivity of $F$ and
boundedness of $a$ imply that $w$ is Lipschitz. Choose $w_\eps$ as in
Lemma~\ref{lem:smoothing}, using periodic convolution if $M=\TT^d$.
Fenchel's inequality gives
\[
 L_F(x,q)+a(x)\geqslant\langle Dw_\eps(x),q\rangle-\eps.
\]
Closedness, together with Proposition~\ref{prop:closed-tests}
when $M=\MM$, gives
\[
 \int_{TM}\langle Dw_\eps(x),q\rangle\,d\tilde\mu(x,q)=0.
\]
Integrating the preceding inequality and then letting $\eps\to0$
proves \eqref{eq:action-pos}.
Note that the integral
is always defined because $L_F+a$ is bounded from below.
\end{proof}

We denote by $\tilde\Mis(G)$ the set of measures
$\tilde\mu\in\cmeasures(\TM)$ satisfying
\begin{equation}\label{eq:MatherG}
 \int_{\TM}L_G(x,q)\,d\tilde\mu(x,q)=-c(G).
\end{equation}
Their projections under $\pi_1$ are termed projected Mather
measures, and form the set $\Mis(G)$.
Since a subsolution exists at level $c_f(G)$,
Lemma~\ref{lem:action} gives
\[
 \int_{\TM}L_G\,d\tilde\mu\geqslant-c_f(G)
 \qquad\hbox{for every $\tilde\mu\in\cmeasures(\TM)$.}
\]
Thus both $\tilde\Mis(G)$ and $\Mis(G)$ are empty whenever
$c(G)>c_f(G)$. If $\tilde\Mis(G)$ is nonempty, then
$c(G)=c_f(G)$ and its elements minimize the action among all
closed probability measures.

The next theorem provides the further information that the lower
bound $-c_f(F)$ is always the infimum of the whole-space action,
whether or not that infimum is attained.

\begin{teorema}\label{thm:free-action}
Under the standing assumptions, for $F=H,G$ one has
\begin{equation*}
\inf_{\tilde\mu\in\cmeasures(\TM)}
\int_{\TM}L_F(x,q)\,d\tilde\mu(x,q)
=-c_f(F).
\end{equation*}
\end{teorema}

\begin{proof}
The lower bound follows from Lemma~\ref{lem:action}, since
$F(x,Dw)=c_f(F)$ admits a subsolution on $\MM$.

For the reverse inequality, fix $a<c_f(F)$ and $x_0\in\MM$.
By Proposition~\ref{prop:metric}, $S_a^F(x_0,x_0)=-\infty$.
By \eqref{eq:actionS}, there exist $T>0$ and a closed Lipschitz curve
$\gamma:[0,T]\to\MM$ such that
\[
 \int_0^T
 \bigl(L_F(\gamma(s),\dot\gamma(s))+a\bigr)\,ds<0.
\]
Define its normalized occupation measure $\tilde\mu_\gamma$
by
\[
 \int_{\TM}f(x,q)\,d\tilde\mu_\gamma(x,q)
 :=
 \frac1T\int_0^T
 f(\gamma(s),\dot\gamma(s))\,ds
 \qquad\hbox{for every $f\in\CC_b(\TM)$.}
\]
This is a probability measure with finite first moment, since
\[
 \int_{\TM}|q|\,d\tilde\mu_\gamma
 =\frac1T\int_0^T|\dot\gamma(s)|\,ds<+\infty.
\]
Moreover, for every $\varphi\in\CC^1(\MM)\cap\Lip(\MM)$,
\[
 \int_{\TM}\langle D\varphi(x),q\rangle
 \,d\tilde\mu_\gamma(x,q)
 =
 \frac{\varphi(\gamma(T))-\varphi(\gamma(0))}{T}
 =0,
\]
because $\gamma$ is closed.
Thus $\tilde\mu_\gamma\in\cmeasures(\TM)$, and
\[
 \inf_{\tilde\mu\in\cmeasures(\TM)}
 \int_{\TM}L_F\,d\tilde\mu
 \leqslant
 \int_{\TM}L_F\,d\tilde\mu_\gamma
 =
 \frac1T\int_0^T L_F(\gamma(s),\dot\gamma(s))\,ds
 <-a.
\]
Letting $a\uparrow c_f(F)$ completes the proof.
\end{proof}

For $H$, the notation $\tilde\Mis(H)$ refers to the periodic
Mather measures, namely the minimizers of
\[
 \min_{\tilde\mu\in\cmeasures(\TTorus)}
 \int_{\TTorus}L_H(x,q)\,d\tilde\mu(x,q)=-c(H);
\]
see, for instance, \cite[Theorem~5.7]{DFIZ1}.
In particular, $\tilde\Mis(H)$ is nonempty.
Their projections onto $\TT^d$ form the set $\Mis(H)$.
Whenever we consider measures for $H$ on the whole space,
we will state this explicitly.

The quotient map $x\mapsto x+\Z^d$ induces the map
$\tilde\pi:\TM\to\TTorus$ defined by
\begin{equation*}
 \tilde\pi(x,q):=(x+\Z^d,q).
\end{equation*}
If $\tilde\mu$ is a finite closed measure on $\TM$, then
$\tilde\pi_\#\tilde\mu$ is closed on $\TTorus$: a smooth periodic
function, lifted to $\MM$, is an admissible test in \eqref{eq:closed}.

We conclude this subsection by recording the part of periodic selection
which will be used below. For $\lambda>0$, let $u_H^\lambda$ be the
periodic solution of
\begin{equation}\tag{$\mathrm{HJ}_H^\lambda$}\label{eq:discountH}
 \lambda u+H(x,Du)=c(H)\qquad\hbox{in $\MM$.}
\end{equation}
By \cite{DFIZ1}, these solutions satisfy
\begin{equation}\label{eq:periodic-lim}
 \|u_H^\lambda-u_H^0\|_\infty\longrightarrow0
 \qquad\hbox{as $\lambda\to0^+$,}
\end{equation}
where $u_H^0$ is a periodic solution of \eqref{eq:criticalH} and
\begin{equation}\tag{$\mathrm{Av}_H$}\label{eq:period-norm}
 \int_{\TT^d}u_H^0(y)\,d\nu(y)\leqslant0
 \qquad\hbox{for every $\nu\in\Mis(H)$.}
\end{equation}
In particular, the periodic discounted solutions $u_H^\lambda$
are uniformly bounded for $0<\lambda\leqslant1$.

\subsection{Discounted equations and minimizing curves}
\label{subsec:disc}

For $F=H,G$, $a\in\R$ and $\lambda>0$, the equation
\begin{equation*}
 \lambda u+F(x,Du)=a\qquad\hbox{in $\MM$}
\end{equation*}
satisfies comparison between bounded subsolutions and supersolutions;
see \cite[Theorem~4.3]{barles}.
Perron's method, applied between constant sub- and supersolutions,
then gives a bounded solution; see \cite[Theorem~2.12]{barles}.
It is unique by comparison and Lipschitz by coercivity.
Comparison with constants gives
\begin{equation}\label{eq:lambdabound}
 \|\lambda u^\lambda\|_\infty\leqslant
 A_F(a):=\sup_{x\in\MM}|F(x,0)-a|.
\end{equation}
It follows from coercivity that a Lipschitz constant for $u^\lambda$
can be chosen independently of $\lambda$.
In particular, we denote by $u_G^\lambda$ the unique bounded solution of
\begin{equation}\tag{$\mathrm{HJ}_G^\lambda$}\label{eq:discountG}
 \lambda u+G(x,Du)=c(G)\qquad\hbox{in $\MM$.}
\end{equation}

We will use the representation formula
\begin{equation}\label{eq:disc-formula}
 u^\lambda(x)=\inf_{\gamma(0)=x}
 \int_{-\infty}^0\e^{\lambda s}
       \big(L_F(\gamma(s),\dot\gamma(s))+a\big)\,ds.
\end{equation}
For the convex superlinear setting considered here, the standard
variational construction also gives the following calibration property;
see \cite[Appendix~2]{DFIZ1}. The construction on the whole space uses
the uniform bounds in (H2).

\begin{prop}\label{prop:calibrated}
For every $\lambda>0$ and $x\in\MM$, there exists a Lipschitz curve
$\gamma_x^\lambda:(-\infty,0]\to\MM$, with $\gamma_x^\lambda(0)=x$,
such that, for every $t>0$,
\begin{equation*}
 u^\lambda(x)={}\e^{-\lambda t}u^\lambda(\gamma_x^\lambda(-t))
 +\int_{-t}^0\e^{\lambda s}
       \big(L_F(\gamma_x^\lambda(s),\dot\gamma_x^\lambda(s))+a\big)\,ds.
\end{equation*}
There exists $\kappa_*>0$, independent of $\lambda$ and $x$, such that
$\|\dot\gamma_x^\lambda\|_\infty\leqslant\kappa_*$.
In particular, this curve realizes the infimum in
\eqref{eq:disc-formula}.
\end{prop}

For completeness, the uniform speed bound can be seen from calibration
and the common Lipschitz bound for $u^\lambda$. Differentiating the
calibration identity along the curve at almost every time gives
\[
 L_F(\gamma,\dot\gamma)+a
 =\lambda u^\lambda(\gamma)
   +\frac{d}{ds}u^\lambda(\gamma)
 \leqslant A_F(a)+\kappa|\dot\gamma|.
\]
Uniform superlinearity of $L_F$ bounds $|\dot\gamma|$. This argument
uses the derivative of the composition $u^\lambda\circ\gamma$, which
exists almost everywhere, and does not require differentiability of
$u^\lambda$ at points of the curve.

A critical solution $v$, not necessarily bounded, likewise has,
through every $x$, a
backward calibrated curve $\gamma_x:(-\infty,0]\to\MM$ with $\gamma_x(0)=x$ satisfying
\begin{equation}\label{eq:crit-calibr}
 v(x)-v(\gamma_x(-t))
 =\int_{-t}^0\big(L_F(\gamma_x(s),\dot\gamma_x(s))+c(F)\big)\,ds
 \qquad(t>0).
\end{equation}
This is a standard consequence of the Lax--Oleinik representation formula for
the solution of $\partial_t U+F(x,D_xU)=0$ with initial datum $v$,
namely $U(t,x)=v(x)-c(F)t$; see, for instance,
\cite[Section~2]{II09}. See also \cite[Theorem~1.2]{Fat24}
for the Tonelli case on noncompact manifolds.
The speed of $\gamma_x$ is bounded by the same superlinearity argument.

\section{Bounds for the discounted solutions}\label{sec:bounds}

We first establish the existence of a bounded critical subsolution
for an arbitrary compactly supported potential. We then prove
local uniform bounds for the discounted solutions when $V\leqslant0$
and global uniform bounds when $V\geqslant0$. These estimates
yield a common compactness conclusion for potentials of constant sign.

\begin{prop}\label{prop:bdd-subsol}
There exists a bounded subsolution $v_G$ of \eqref{eq:criticalG}.
It can be chosen so that $v_G\leqslant0$ on $\MM$.
\end{prop}

\begin{proof}
Since $c(G)\geqslant c_f(G)$, there exists a subsolution $w$ of
\eqref{eq:criticalG}, which we normalize by $w(0)=0$. Set
\[
 S(x):=S_{c(G)}^G(0,x),\qquad x\in\MM.
\]
We first show that $S$ is bounded from below. Choose $r>0$ so that
$\supp(V)\subset B_r$, and let $\kappa$ be a Lipschitz constant for
$w$. If $x\in\overline B_r$, domination gives
$S(x)\geqslant w(x)\geqslant-\kappa r$.

Let now $x\notin\overline B_r$, and consider a curve from $0$ to $x$
in the action representation \eqref{eq:actionS}. Denote by $z$ its
last exit point from $\overline B_r$. The action of the part ending
at $z$ is bounded from below by $w(z)-w(0)$. On the remaining part,
$L_G=L_H$ and $c(G)\geqslant c(H)$, so domination by the periodic
critical solution $u_H^0$ yields the lower bound
$u_H^0(x)-u_H^0(z)$. Thus
\[
 S(x)\geqslant-\kappa r-2\|u_H^0\|_\infty
 \qquad\hbox{for every $x\in\MM$.}
\]

Choose $k>0$ sufficiently large that $S<u_H^0+k$ on a neighborhood
of $\supp(V)$, and define
\[
 \hat v_G(x):=\min\{S(x),u_H^0(x)+k\},\qquad x\in\MM.
\]
Near $\supp(V)$ this function agrees with $S$. Away from $\supp(V)$,
both functions in the minimum are subsolutions at level $c(G)$.
Convexity therefore implies that $\hat v_G$ is a critical subsolution
on all of $\MM$. It is bounded from above by $u_H^0+k$, and from
below by the preceding estimate for $S$ and the boundedness of $u_H^0$.
Subtracting $\sup_{\MM}\hat v_G$ gives the required function $v_G$.
\end{proof}

Fix a bounded critical subsolution $v_G\leqslant0$ as above.
It is a subsolution of the discounted equation, so comparison gives
\begin{equation}\label{eq:discount-lower}
 v_G\leqslant u_G^\lambda\qquad\hbox{in $\MM$ for every $\lambda>0$.}
\end{equation}
There is also a common Lipschitz bound, without any sign assumption
on $V$. Indeed, by \eqref{eq:lambdabound}, every $u_G^\lambda$
is a subsolution of
\[
 G(x,Du)=c(G)+A_G(c(G))\qquad\hbox{in $\MM$},
\]
and the bound follows from uniform coercivity.

We first consider nonpositive potentials. If $V\leqslant0$,
comparison gives
\begin{equation}\label{eq:comparison}
 v_G\leqslant u_G^\lambda
 \leqslant u_H^\lambda+\frac{c(G)-c(H)}{\lambda}
 \qquad\hbox{in $\MM$.}
\end{equation}
Indeed, $v_G\leqslant0$ is a subsolution of \eqref{eq:discountG},
while the function on the right is a supersolution since $G\geqslant H$.
The upper bound in \eqref{eq:comparison} is uniform in $\lambda$
only if the critical values agree. The next proposition supplies the
local estimate needed in both regimes.

\begin{prop}\label{prop:compact}
Assume that $V\leqslant0$. The family
$\{u_G^\lambda:0<\lambda\leqslant1\}$ is equi--Lipschitz, has a common
constant lower bound, and is locally uniformly bounded from above.
More precisely, let $v_G\leqslant0$ be a bounded critical subsolution
as in Proposition~\ref{prop:bdd-subsol}, and let $\kappa>0$ be a common
Lipschitz constant for the functions $u_G^\lambda$. There exists
a constant $C>0$ such that
\begin{equation}\label{eq:localupper}
 -\|v_G\|_\infty\leqslant u_G^\lambda(x)
 \leqslant C+\kappa|x|
 \qquad\hbox{for every $x\in\MM$ and $0<\lambda\leqslant1$.}
\end{equation}

Consequently this family is precompact for local uniform convergence,
and $\lambda u_G^\lambda\ucv0$ as $\lambda\to0^+$.
When $c(G)=c(H)$, it is uniformly bounded on $\MM$.
\end{prop}
\begin{proof}
Let us set
\[
 C_H:=\sup_{0<\lambda\leqslant1}\|u_H^\lambda\|_\infty<+\infty,
\]
and choose $r>0$ such that $\supp(V)\subset B_r$.
The common Lipschitz bound and the lower bound
$u_G^\lambda\geqslant-\|v_G\|_\infty$ were established above.
It remains to prove the local upper bound.

Fix $\lambda\in(0,1]$. We first observe that
\begin{equation}\label{eq:infdiscount}
 \inf_{\MM}u_G^\lambda\leqslant0.
\end{equation}
Indeed, if this infimum were a positive number $m$, the equation
for $u_G^\lambda$ would give
\[
 G(x,Du_G^\lambda)=c(G)-\lambda u_G^\lambda(x)
 \leqslant c(G)-\lambda m\qquad\hbox{in $\MM$.}
\]
Thus $u_G^\lambda$ would be a bounded subsolution at a level
strictly smaller than $c(G)$, contradicting \eqref{eq:cvalue}.
The infimum need not be attained for this argument.

Write $\delta=c(G)-c(H)\geqslant0$ for the duration of this proof.
We next estimate $u_G^\lambda$ outside the perturbation region.
Fix $x\notin\overline B_r$ and a backward calibrated curve
$\gamma=\gamma_x^\lambda$ for $u_G^\lambda$, as in
Proposition~\ref{prop:calibrated}. Define
\[
 t:=\inf\{s\geqslant0:\gamma(-s)\in\overline B_r\},
 \qquad\inf\emptyset:=+\infty.
\]
If $t<+\infty$, then $t>0$, $z:=\gamma(-t)\in\partial B_r$,
and $\gamma(s)\notin\overline B_r$ for $-t<s\leqslant0$.
Since $\supp(V)\subset B_r$, one has $L_G=L_H$ along this
portion of the curve. Calibration gives
\begin{align*}
 u_G^\lambda(x)
 ={}&\e^{-\lambda t}u_G^\lambda(z)
 +\int_{-t}^0\e^{\lambda s}
       \big(L_H(\gamma(s),\dot\gamma(s))+c(H)\big)\,ds\\
 &+\frac{\delta}{\lambda}(1-\e^{-\lambda t}).
\end{align*}
On the other hand, the same curve is an admissible competitor for
the periodic discounted problem. Its dynamic programming inequality
reads
\[
 u_H^\lambda(x)\leqslant\e^{-\lambda t}u_H^\lambda(z)
 +\int_{-t}^0\e^{\lambda s}
       \big(L_H(\gamma(s),\dot\gamma(s))+c(H)\big)\,ds.
\]
Subtracting this inequality from the preceding identity yields
\[
 u_G^\lambda(x)-u_H^\lambda(x)
 \geqslant\e^{-\lambda t}
       \big(u_G^\lambda(z)-u_H^\lambda(z)\big)
       +(1-\e^{-\lambda t})\frac{\delta}{\lambda}.
\]
The right-hand side is a convex combination of the boundary
difference at $z$ and $\delta/\lambda$.

If $t=+\infty$, the same calculation applies on $[-T,0]$ for
every $T>0$, with $\gamma(-T)$ in place of $z$. Both discounted
solutions are bounded for the fixed $\lambda$ under consideration,
so their difference multiplied by $\e^{-\lambda T}$ tends to zero
as $T\to+\infty$. We obtain
$u_G^\lambda(x)-u_H^\lambda(x)\geqslant\delta/\lambda$.
In either case,
\begin{equation}\label{eq:exterior}
\begin{split}
 u_G^\lambda(x)\geqslant u_H^\lambda(x)+
 \min\left\{\frac{\delta}{\lambda},\,
 \min_{z\in\partial B_r}
       \big(u_G^\lambda(z)-u_H^\lambda(z)\big)\right\}.
\end{split}
\end{equation}

We now show that
\[
 m_\lambda:=\min_{\overline B_r}u_G^\lambda\leqslant2C_H.
\]
Suppose, to the contrary, that $m_\lambda>2C_H$, and set
\[
 b_\lambda:=\min_{z\in\partial B_r}
       \big(u_G^\lambda(z)-u_H^\lambda(z)\big).
\]
Since $\|u_H^\lambda\|_\infty\leqslant C_H$, we have
$b_\lambda\geqslant m_\lambda-C_H>C_H$.
If also $\delta/\lambda>C_H$, \eqref{eq:exterior} would give
the uniform positive bound
\[
 u_G^\lambda(x)\geqslant
 -C_H+\min\{\delta/\lambda,b_\lambda\}>0
 \qquad\hbox{for every $x\notin\overline B_r$.}
\]
Together with $u_G^\lambda\geqslant m_\lambda>0$ on
$\overline B_r$, this would contradict \eqref{eq:infdiscount}.
We must therefore have $\delta/\lambda\leqslant C_H$.
But the upper comparison in \eqref{eq:comparison} then implies
\[
 u_G^\lambda\leqslant u_H^\lambda+\frac{\delta}{\lambda}
 \leqslant2C_H\qquad\hbox{in $\MM$,}
\]
contradicting $m_\lambda>2C_H$. This proves the claim without
distinguishing whether $\delta$ is zero or positive.

Choose $z_\lambda\in\overline B_r$ with
$u_G^\lambda(z_\lambda)=m_\lambda$. The common Lipschitz bound gives
\[
 u_G^\lambda(x)\leqslant u_G^\lambda(z_\lambda)
       +\kappa|x-z_\lambda|
 \leqslant2C_H+\kappa(|x|+r).
\]
This proves \eqref{eq:localupper} with $C:=2C_H+\kappa r$.
On each compact set the family is
thus uniformly bounded and equi--Lipschitz. The Ascoli--Arzel\`a
theorem and a diagonal extraction give precompactness for local
uniform convergence. The same local bounds imply
$\sup_K|\lambda u_G^\lambda|\to0$ for every compact
$K\subset\MM$. Finally, if $c(G)=c(H)$, \eqref{eq:comparison}
gives the global bounds
$-\|v_G\|_\infty\leqslant u_G^\lambda\leqslant C_H$.
\end{proof}

We now turn to nonnegative potentials. In this case
$c(G)=c(H)$ by Proposition~\ref{prop:corder}, and every periodic
critical solution of $H$ is a bounded critical subsolution of $G$.
We first use the compact support of $V$ to construct a bounded
critical solution above it. For the next theorem and its proof,
we write $c:=c(G)=c(H)$.

\begin{teorema}\label{thm:pos-existence}
Suppose $V\geqslant0$. For every periodic solution $u_H$ of
$H(x,Du_H)=c$, there exist a constant $C>0$ and a bounded
solution $w_G$ of $G(x,Dw_G)=c$ such that
\[
 u_H\leqslant w_G\leqslant u_H+C
 \qquad\hbox{in $\MM$.}
\]
Moreover, $w_G$ is the least bounded critical solution
of $G$ lying above $u_H$.
\end{teorema}

\begin{proof}
Let $(\S_G(t))_{t\geqslant0}$ denote the Lax--Oleinik semigroup
associated with the running cost $L_G+c$, defined for
uniformly continuous functions $\varphi:\MM\to\R$ by
\[
 (\S_G(t)\varphi)(x)
 :=
 \inf_{\substack{\gamma\in\mathrm{AC}([-t,0];\MM)\\
                  \gamma(0)=x}}
 \left\{
 \varphi(\gamma(-t))
 +\int_{-t}^0
 \bigl(L_G(\gamma(s),\dot\gamma(s))+c\bigr)\,ds
 \right\},
 \qquad t>0,
\]
with $\S_G(0)\varphi=\varphi$.
For its standard properties and their relation to viscosity
solutions, see, for instance, \cite[Section~2]{II09} and
\cite[Section~2.2]{DZ10}.
Set $U(t,\cdot):=\S_G(t)u_H$.
Since $G\leqslant H$, $u_H$ is a critical subsolution for $G$.
The domination inequality and the order-preserving semigroup
property give
\begin{equation}\label{eq:pos-LO-monotone}
 u_H\leqslant U(t,\cdot),\qquad
 U(t,\cdot)\leqslant U(t+s,\cdot)=\S_G(t)U(s,\cdot)
 \quad(s,t>0).
\end{equation}

We claim that there exists a constant $C>0$, independent of
$t$ and $x$, such that
\begin{equation}\label{eq:pre-LO-bound}
 U(t,\cdot)\leqslant u_H+C
 \quad\hbox{on $\MM$,}
 \qquad\hbox{for every $t>0$.}
\end{equation}
Choose $r>0$ with $\operatorname{supp}V\subset B_r(0)$ and,
for $R\geqslant0$, set
\[
 M_R:=\sup_{y\in\MM,\,|q|\leqslant R}|L_G(y,q)+c|<+\infty.
\]
Backward calibrated curves for $u_H$ exist through every point,
are calibrated on every subinterval, and have a common speed
bound $\kappa_*>0$; see \eqref{eq:crit-calibr}.
Their integer translates remain calibrated by periodicity.
Outside $\overline B_r(0)$ one has $L_G=L_H$.

First suppose that there exists such a curve $\gamma_*$
with bounded image, and set
\[
 a:=\gamma_*(0),\qquad
 A:=\overline{\gamma_*((-\infty,0])}.
\]
Only finitely many $z\in\Z^d$ satisfy
$(A+z)\cap\overline B_r(0)\ne\emptyset$.
Consequently, there exists $D<+\infty$ such that, for every
$x\in\MM$, one can choose $z\in\Z^d$ with
\[
 (A+z)\cap\overline B_r(0)=\emptyset,\qquad
 |a+z-x|\leqslant D.
\]
For $t>1$, follow $\gamma_*+z$ for $t-1$ units of time,
ending at $a+z$, and then join $a+z$ to $x$ by a straight
segment traversed in one unit of time.
Calibration along the first arc and the speed bound $D$
on the segment give
\begin{equation}\label{eq:pos-LO-bound-one}
 U(t,x)\leqslant u_H(a+z)+M_D
 \leqslant u_H(x)+2\|u_H\|_\infty+M_D.
\end{equation}
By \eqref{eq:pos-LO-monotone}, the same bound holds for
$0<t\leqslant1$.

Suppose instead that every backward calibrated curve has
unbounded image. Fix $x\in\MM$, choose such a curve $\gamma$
ending at $x$, and take a sequence $t_n\to+\infty$ such that
$y_n:=\gamma(-t_n)$ satisfies $|y_n|\to+\infty$.

If $\gamma([-t_n,0])$ avoids $\overline B_r(0)$, set
$\hat t_n:=t_n$. Calibration then gives
\begin{equation}\label{eq:pos-LO-bound-easy}
 U(\hat t_n,x)\leqslant u_H(x).
\end{equation}

Otherwise, let
$-t_n\leqslant\sigma_n\leqslant\tau_n\leqslant0$
be the first and last contact times of $\gamma|_{[-t_n,0]}$
with $\overline B_r(0)$, and set
\[
 a_n:=\gamma(\sigma_n),\qquad b_n:=\gamma(\tau_n).
\]
Retain the portions of $\gamma$ before $\sigma_n$ and after
$\tau_n$, replacing the entire portion between them by the
straight segment from $a_n$ to $b_n$ traversed in one unit
of time.
The retained arcs lie outside $\overline B_r(0)$ apart from
their endpoints and are calibrated for $H$.
The connecting segment has speed at most $2r$.
Thus, setting $\hat t_n:=t_n-(\tau_n-\sigma_n)+1$ for the
new duration, we obtain
\begin{equation}\label{eq:pos-LO-bound-two}
 U(\hat t_n,x)\leqslant
 u_H(x)+u_H(a_n)-u_H(b_n)+M_{2r}
 \leqslant u_H(x)+2\|u_H\|_\infty+M_{2r}.
\end{equation}
Moreover, the retained arc $\gamma|_{[-t_n,\sigma_n]}$
joins $y_n$ to $a_n\in\overline B_r(0)$ and has speed at
most $\kappa_*$, so
\[
 \hat t_n
 =(t_n+\sigma_n)+1-\tau_n
 \geqslant t_n+\sigma_n
 \geqslant
 \frac{\operatorname{dist}(y_n,\overline B_r(0))}{\kappa_*}
 \longrightarrow+\infty.
\]

We have therefore constructed a sequence $\hat t_n\to+\infty$
such that, for each $n$, either \eqref{eq:pos-LO-bound-easy}
or \eqref{eq:pos-LO-bound-two} holds.
Monotonicity in \eqref{eq:pos-LO-monotone} extends the common
upper bound to every $t>0$.
Together with \eqref{eq:pos-LO-bound-one}, this proves
\eqref{eq:pre-LO-bound}.

It follows that $U(t,x)$ increases to a function $w_G$ with
$u_H\leqslant w_G\leqslant u_H+C$. Moreover,
\[
 U(t,\cdot)\leqslant U(t+s,\cdot)=\S_G(s)U(t,\cdot)
 \qquad\hbox{for every $s,t>0$,}
\]
so every $U(t,\cdot)$ is a critical subsolution.
Uniform coercivity of $G$ therefore gives a common Lipschitz
constant for this family, and its convergence to $w_G$ is
locally uniform.

A standard semigroup argument shows that $w_G$ is a critical
solution. We include the argument for the reader's convenience.
For each fixed $s>0$, the Lax--Oleinik operator $\S_G(s)$
is continuous with respect to locally uniform convergence
on uniformly bounded, equi--Lipschitz families.
Indeed, superlinearity of $L_G$ and the common bound on the
initial data confine the starting points of nearly minimizing
curves to a common bounded set when the endpoints range over
a compact set. Passing to the limit as $t\to+\infty$ in the
semigroup identity gives
\[
 \S_G(s)w_G
 =\lim_{t\to+\infty}\S_G(s)U(t,\cdot)
 =\lim_{t\to+\infty}U(t+s,\cdot)
 =w_G.
\]
Thus $w_G$ is a fixed point of the semigroup, hence a
stationary viscosity solution of the evolutionary equation
$\partial_t U+G(x,D_xU)=c$. Consequently,
$G(x,Dw_G)=c$ in $\MM$.

Finally, if $v$ is any bounded critical solution with
$v\geqslant u_H$, order preservation and stationarity imply
\[
 U(t,\cdot)=\S_G(t)u_H\leqslant\S_G(t)v=v.
\]
Letting $t\to+\infty$ yields $w_G\leqslant v$.
This proves that $w_G$ is the least bounded critical solution
of $G$ lying above $u_H$.
\end{proof}

The bounded critical solutions just constructed provide supersolutions
of the discounted equation after adding a suitable constant. Together
with comparison with the periodic discounted solutions, this gives
global uniform bounds.

\begin{prop}\label{prop:pos-bdd}
Suppose $V\geqslant0$. Then the family
$\{u_G^\lambda:0<\lambda\leqslant1\}$ is uniformly bounded on
$\MM$, equi--Lipschitz and precompact for local uniform convergence.
Moreover, $\lambda u_G^\lambda\to0$ uniformly on $\MM$
as $\lambda\to0^+$.
\end{prop}

\begin{proof}
Let $w_G$ be a bounded critical solution given by
Theorem~\ref{thm:pos-existence}. Adding a constant, we may assume
$w_G\geqslant0$. It is then a supersolution of the discounted
equation. Since $G\leqslant H$ and $c(G)=c(H)$, comparison gives
\begin{equation}\label{eq:pos-compare}
 u_H^\lambda\leqslant u_G^\lambda\leqslant w_G
 \qquad\hbox{in $\MM$.}
\end{equation}
The periodic discounted solutions $u_H^\lambda$ are uniformly
bounded for $0<\lambda\leqslant1$, so the same is true of
$u_G^\lambda$. Together with the common Lipschitz bound established
above, this gives local precompactness by the Arzel\`a--Ascoli
theorem. The uniform bound also implies
$\lambda u_G^\lambda\to0$ uniformly on $\MM$.
\end{proof}

We can now collect the properties of accumulation points that
will be used in the selection arguments. Write
\begin{equation*}
 \omegasetG:=\big\{u\in\CC(\MM):
  u_G^{\lambda_n}\ucv u\hbox{ for some }\lambda_n\to0^+\big\}
\end{equation*}
for the set of accumulation points.

\begin{cor}\label{cor:clusters}
Assume that $V$ has constant sign. Then $\omegasetG$
is nonempty, and every $u\in\omegasetG$ is a Lipschitz solution
of \eqref{eq:criticalG} satisfying $u\geqslant v_G$.
Moreover,
\[
\begin{aligned}
 u&\leqslant u_H^0
 &&\quad\hbox{if $V\leqslant0$\quad and\quad $c(G)=c(H)$,}\\
 u&\geqslant u_H^0
 &&\quad\hbox{if $V\geqslant0$.}
\end{aligned}
\]
In particular, every accumulation point is bounded when
$c(G)=c(H)$.
\end{cor}

\begin{proof}
Propositions~\ref{prop:compact} and \ref{prop:pos-bdd} give local
precompactness and $\lambda u_G^\lambda\ucv0$ for either sign of $V$.
Thus $\omegasetG$ is nonempty, and viscosity stability shows that
each of its elements solves \eqref{eq:criticalG}.
The common Lipschitz bound passes to the limit, as does the
lower bound in \eqref{eq:discount-lower}.
The comparisons with $u_H^0$ follow by passing to the limit in
\eqref{eq:comparison} or \eqref{eq:pos-compare}, using
\eqref{eq:periodic-lim}. Finally, when $c(G)=c(H)$, both
propositions give global uniform bounds, which pass to every
accumulation point.
\end{proof}

\section{Discounted measures and loss of mass}\label{sec:measures}

\subsection{Discounted measures and integral inequalities}

Fix $x\in\MM$. For every $\lambda>0$, choose a calibrated curve
$\gamma_x^\lambda$ for $u_G^\lambda$ as in
Proposition~\ref{prop:calibrated}. We define a probability measure
$\tilde\mu_x^\lambda$ on $\TM$ by setting
\begin{equation}\label{eq:disc-measure}
 \int_{\TM}f(y,q)\,d\tilde\mu_x^\lambda(y,q)
 :=\lambda\int_{-\infty}^0\e^{\lambda s}
 f\big(\gamma_x^\lambda(s),\dot\gamma_x^\lambda(s)\big)\,ds
\end{equation}
for every $f\in\CC_b(\TM)$. Its projection onto $\MM$ is denoted by
$\mu_x^\lambda=\pi_{1\#}\tilde\mu_x^\lambda$. By the uniform speed
bound, all these measures are supported in
$\MM\times\overline B_{\kappa_*}$. This set is not compact; only the
velocity variable is uniformly bounded.

The speed bound also gives, for each fixed $\lambda>0$,
\begin{equation}\label{eq:disc-spatial-moment}
 \int_{\MM}|y-x|\,d\mu_x^\lambda(y)
 \leqslant\lambda\kappa_*\int_{-\infty}^0
             \e^{\lambda s}|s|\,ds
 =\frac{\kappa_*}{\lambda}.
\end{equation}
In particular, every Lipschitz function is integrable against
$\mu_x^\lambda$.
Integration by parts gives, for every
$\varphi\in\CC^1(\MM)\cap\Lip(\MM)$,
\begin{equation}\label{eq:closeddefect}
 \int_{\TM}\langle D\varphi(y),q\rangle\,d\tilde\mu_x^\lambda(y,q)
 =\lambda\left(\varphi(x)-\int_{\MM}\varphi(y)\,d\mu_x^\lambda(y)\right).
\end{equation}
Indeed, $\varphi(\gamma_x^\lambda(s))$ grows at most linearly
as $s\to-\infty$, so the boundary term vanishes.
For bounded $\varphi$, the absolute value of the left-hand side
is at most $2\lambda\|\varphi\|_\infty$.
The minimizing property also yields
\begin{equation}\label{eq:disc-action}
 \int_{\TM}\big(L_G+c(G)\big)\,d\tilde\mu_x^\lambda
 =\lambda u_G^\lambda(x).
\end{equation}

The following two inequalities require no sign assumption on $V$.

\begin{prop}\label{prop:integrals}
Let $\lambda>0$ and $x\in\MM$.
\begin{itemize}
\item[(i)] For every critical subsolution $w$ of $G$, possibly unbounded,
\begin{equation}\label{eq:subsol-int}
 u_G^\lambda(x)\geqslant w(x)
               -\int_{\MM}w(y)\,d\mu_x^\lambda(y).
\end{equation}
\item[(ii)] For every $\mu\in\Mis(G)$,
\begin{equation}\label{eq:disc-norm}
 \int_{\MM}u_G^\lambda(y)\,d\mu(y)\leqslant0.
\end{equation}
\end{itemize}
\end{prop}

\begin{proof}
To prove (i), fix $\eps>0$ and let $w_\eps$ be given by
Lemma~\ref{lem:smoothing}. Fenchel's inequality gives
\[
 L_G(y,q)+c(G)\geqslant
 \langle Dw_\eps(y),q\rangle-\eps.
\]
Integrating against $\tilde\mu_x^\lambda$ and using
\eqref{eq:closeddefect} and \eqref{eq:disc-action}, we obtain
\[
 \lambda u_G^\lambda(x)
 \geqslant\lambda\left(w_\eps(x)
       -\int_{\MM}w_\eps(y)\,d\mu_x^\lambda(y)\right)-\eps.
\]
The integral is finite by \eqref{eq:disc-spatial-moment}.
Divide by $\lambda$ and let $\eps\to0$, using
$\|w_\eps-w\|_\infty\leqslant\eps$ and the unit mass of
$\mu_x^\lambda$, to obtain \eqref{eq:subsol-int}.

For (ii), let $\tilde\mu\in\tilde\Mis(G)$ project to $\mu$.
Apply Lemma~\ref{lem:action} to $u_G^\lambda$ with the bounded,
uniformly continuous right-hand side $a(y):=c(G)-\lambda u_G^\lambda(y)$.
Using \eqref{eq:MatherG}, we obtain
\[
 0\leqslant\int_{\TM}
       \bigl(L_G(y,q)+c(G)-\lambda u_G^\lambda(y)\bigr)\,d\tilde\mu(y,q)
 =-\lambda\int_{\MM}u_G^\lambda\,d\mu.
\]
This proves \eqref{eq:disc-norm}.
\end{proof}

The common lower bound in \eqref{eq:discount-lower} allows us to
pass to the limit in \eqref{eq:disc-norm} by Fatou's lemma.
This argument requires no sign assumption on $V$ and applies
whenever an accumulation point exists.

\begin{prop}\label{prop:cluster-norm}
Every $u\in\omegasetG$ is integrable against every $\mu\in\Mis(G)$
and satisfies
\begin{equation*}
 \int_{\MM}u(y)\,d\mu(y)\leqslant0.
\end{equation*}
\end{prop}

\begin{proof}
Choose $\lambda_n\to0^+$ such that $u_G^{\lambda_n}\ucv u$.
By \eqref{eq:discount-lower}, the functions $u_G^{\lambda_n}$
have the common constant lower bound $-\|v_G\|_\infty$.
Fix $\mu\in\Mis(G)$. Applying Fatou's lemma to
$u_G^{\lambda_n}+\|v_G\|_\infty$ and using \eqref{eq:disc-norm}
gives
\[
 \int_{\MM}u\,d\mu
 \leqslant\liminf_{n\to\infty}\int_{\MM}u_G^{\lambda_n}\,d\mu
 \leqslant0.
\]
Since $u\geqslant-\|v_G\|_\infty$ and $\mu$ is a probability
measure, this also proves that $u$ is integrable against $\mu$.
\end{proof}

\subsection{The limiting decomposition}\label{subsec:decomp}

We now compare limits of discounted measures before and after applying
$\tilde\pi:\TM\to\TTorus$. We denote vague convergence of measures
by $\weakst$ and narrow convergence by $\weakcv$, meaning convergence
against compactly supported and bounded continuous functions,
respectively. These conventions apply on either space.
Vague limits are allowed to have mass smaller than one.

The next proposition is stated for general measures because the same
argument will also apply to measures constructed from critical calibrated
curves. Neither a sign condition on $V$ nor a tightness assumption is
needed in its statement.

\begin{prop}\label{prop:decomp}
Let $(\tilde\mu_n)_n$ be probability measures on $\TM$, supported in
$\MM\times\overline B_K$ for some $K>0$, and assume that
\begin{equation}\label{eq:asym-closed}
 \int_{\TM}\langle D\varphi(y),q\rangle\,d\tilde\mu_n(y,q)
 \longrightarrow0
 \quad\hbox{for every $\varphi\in\CC_b^1(\MM)\cap\Lip(\MM)$,}
\end{equation}
and
\begin{equation}\label{eq:asym-action}
 \int_{\TM}\big(L_G+c(G)\big)\,d\tilde\mu_n\longrightarrow0.
\end{equation}
After passing to a subsequence, there are a nonnegative finite measure
$\tilde\mu$ on $\TM$ and a probability measure $\tilde\nu$ on
$\TTorus$ such that
\begin{equation}\label{eq:mu-limits}
 \tilde\mu_n\weakst\tilde\mu\quad\hbox{vaguely in $\TM$,}
 \qquad
 \tilde\pi_\#\tilde\mu_n\weakcv\tilde\nu
 \quad\hbox{narrowly on $\TTorus$.}
\end{equation}
The measure $\tilde\mu$ is closed, and
\begin{equation*}
 \tilde\nu_\infty:=\tilde\nu-\tilde\pi_\#\tilde\mu
\end{equation*}
is a nonnegative closed measure of mass $1-\tilde\mu(\TM)$.
Moreover, the following action identity holds:
\begin{equation}\label{eq:act-balance}
\begin{split}
 0={}&\int_{\TM}\big(L_G+c(G)\big)\,d\tilde\mu
       +\int_{\TTorus}\big(L_H+c(H)\big)\,d\tilde\nu_\infty\\
    &+\big(c(G)-c(H)\big)\tilde\nu_\infty(\TTorus).
\end{split}
\end{equation}
Its three terms are nonnegative and vanish separately. In particular,
\begin{equation}\label{eq:norm-limits}
\begin{aligned}
 \frac{\tilde\mu}{\tilde\mu(\TM)}&\in\tilde\Mis(G)
 &&\hbox{if $\tilde\mu(\TM)>0$,}\\[4pt]
 \frac{\tilde\nu_\infty}{\tilde\nu_\infty(\TTorus)}
 &\in\tilde\Mis(H)
 &&\hbox{if $\tilde\nu_\infty(\TTorus)>0$.}
\end{aligned}
\end{equation}
If $c(G)>c(H)$, then $\tilde\nu_\infty=0$ and
$\tilde\mu_n\weakcv\tilde\mu\in\tilde\Mis(G)$ narrowly in $\TM$.
\end{prop}

\begin{proof}
The extraction in \eqref{eq:mu-limits} follows from weak-* compactness
of measures of bounded mass and from compactness of
$\TT^d\times\overline B_K$. The first limit has mass at most one, and
both limits have velocities in $\overline B_K$.

Note also that, if $\tilde\mu_n\weakst\tilde\mu$ vaguely in $\TM$, then
$\pi_{1\#}\tilde\mu_n\weakst\pi_{1\#}\tilde\mu$ vaguely in $\MM$,
thanks to the common velocity bound.
Indeed, choose $\vartheta\in\CC_c(\R^d)$ equal to one on
$\overline B_K$. For every $f\in\CC_c(\MM)$, the function
$(x,q)\mapsto f(x)\vartheta(q)$ has compact support in $\TM$,
and
\[
\begin{aligned}
 \int_{\MM}f\,d(\pi_{1\#}\tilde\mu_n)
 &=\int_{\TM}f(x)\vartheta(q)\,d\tilde\mu_n(x,q)\\
 &\xrightarrow[n\to+\infty]{}
 \int_{\TM}f(x)\vartheta(q)\,d\tilde\mu(x,q)
 =\int_{\MM}f\,d(\pi_{1\#}\tilde\mu).
\end{aligned}
\]

Let $\chi_R\in\CC_c^\infty(\MM)$ satisfy
\[
 0\leqslant\chi_R\leqslant1,\qquad
 \chi_R=1\hbox{ on }B_R,\qquad
 \chi_R=0\hbox{ outside }B_{2R}.
\]
These functions increase pointwise to one
as $R\to+\infty$ along $R\in\N$.
If $f\in\CC_c(\TTorus)$ is nonnegative, vague convergence gives
\begin{align*}
 \int_{\TM}\chi_R(y)f(\tilde\pi(y,q))\,d\tilde\mu
 &=\lim_n\int_{\TM}\chi_R(y)f(\tilde\pi(y,q))\,d\tilde\mu_n\\
 &\leqslant\lim_n\int_{\TTorus}f\,d(\tilde\pi_\#\tilde\mu_n)
 =\int_{\TTorus}f\,d\tilde\nu.
\end{align*}
Letting $R\to+\infty$ along $R\in\N$ and using the monotone
convergence theorem proves $\tilde\pi_\#\tilde\mu\leqslant\tilde\nu$.
Hence $\tilde\nu_\infty$ is nonnegative and has the stated mass.

For $\varphi\in\CC_c^1(\MM)$, passing to the vague limit in
\eqref{eq:asym-closed}, using the common velocity bound, shows that
\[
 \int_{\TM}\langle D\varphi(y),q\rangle\,d\tilde\mu(y,q)=0.
\]
Since $\tilde\mu$ has bounded velocity support,
Proposition~\ref{prop:closed-tests} implies that it is closed.

Taking periodic tests in \eqref{eq:asym-closed} shows that
$\tilde\nu$ is closed. The measure $\tilde\pi_\#\tilde\mu$ is
closed as well, so their difference $\tilde\nu_\infty$ is closed.

It remains to identify the actions. First note that
\begin{align*}
 \int_{\TM}\bigl(L_G+c(G)\bigr)\,d\tilde\mu_n
 &=
 \int_{\TM}\bigl(L_H+c(G)\bigr)\,d\tilde\mu_n
 +\int_{\TM}V\,d\tilde\mu_n\\
 &=
 \int_{\TTorus}\bigl(L_H+c(G)\bigr)\,
       d(\tilde\pi_\#\tilde\mu_n)
 +\int_{\MM}V(y)\,d(\pi_{1\#}\tilde\mu_n)(y).
\end{align*}
Since $V$ is continuous and compactly supported, vague convergence
of the spatial projections gives convergence of the last integral.
The other integral converges by \eqref{eq:mu-limits}, since $L_H$
is continuous and the velocities lie in the common compact set
$\overline B_K$.
Passing to the limit as $n\to+\infty$ and using
\eqref{eq:asym-action}, we obtain
\begin{align*}
 0
 &=\int_{\TTorus}\bigl(L_H+c(G)\bigr)\,d\tilde\nu
       +\int_{\MM}V(y)\,d(\pi_{1\#}\tilde\mu)(y)\\
 &=\int_{\TM}\bigl(L_G+c(G)\bigr)\,d\tilde\mu
       +\int_{\TTorus}\bigl(L_H+c(H)\bigr)\,d\tilde\nu_\infty\\
 &\qquad+\bigl(c(G)-c(H)\bigr)\tilde\nu_\infty(\TTorus).
\end{align*}
This is precisely \eqref{eq:act-balance}. Both action integrals
are nonnegative by Lemma~\ref{lem:action}, applied to a critical
subsolution of $G$ and a periodic critical subsolution of $H$,
respectively. The last term is nonnegative by
Proposition~\ref{prop:corder}.
They must therefore all vanish, which proves
\eqref{eq:norm-limits}.

If $c(G)>c(H)$, we infer
$\tilde\nu_\infty(\TTorus)=0$, so $\tilde\mu(\TM)=1$.
Vague convergence of probability measures to a probability measure
implies narrow convergence. This proves the last assertion.
\end{proof}

\begin{oss}\label{rem:decomp}
The measure which minimizes the periodic action is
$\tilde\nu_\infty$, after normalization, rather than necessarily the
full limit $\tilde\nu$ on $\TTorus$. Indeed, the identity in the proof gives
\[
 \int_{\TTorus}\big(L_H+c(H)\big)\,d\tilde\nu
 =-\int_{\MM}V\,d(\pi_{1\#}\tilde\mu)-\big(c(G)-c(H)\big).
\]
In particular, even when $c(G)=c(H)$, a negative potential can leave a
positive periodic action in the part retained in the whole space.
\end{oss}

For discounted measures, the only additional condition needed
to apply the decomposition is the vanishing of their action in
\eqref{eq:disc-action}. The bounds in Section~\ref{sec:bounds}
ensure this for potentials of either constant sign.

\begin{lemma}\label{lem:discounted-decomp}
Fix $x\in\MM$ and a sequence $\lambda_n\to0^+$ such that
$\lambda_n u_G^{\lambda_n}(x)\to0$. Then the discounted measures
$(\tilde\mu_x^{\lambda_n})_n$ defined by \eqref{eq:disc-measure}
satisfy the hypotheses of Proposition~\ref{prop:decomp}.
In particular, this holds for every sequence $\lambda_n\to0^+$
when $V$ has constant sign, and along any sequence for which
$u_G^{\lambda_n}$ converges locally uniformly.
\end{lemma}

\begin{proof}
The measures $\tilde\mu_x^{\lambda_n}$ are probability measures.
By the uniform speed bound in Proposition~\ref{prop:calibrated},
their supports are contained in $\MM\times\overline B_{\kappa_*}$,
where $\kappa_*>0$ is independent of $n$.

For every $\varphi\in\CC_b^1(\MM)\cap\Lip(\MM)$,
identity \eqref{eq:closeddefect} gives
\[
 \left|\int_{\TM}\langle D\varphi(y),q\rangle\,
             d\tilde\mu_x^{\lambda_n}(y,q)\right|
 \leqslant2\lambda_n\|\varphi\|_\infty
 \longrightarrow0.
\]
Thus \eqref{eq:asym-closed} holds. Finally, by
\eqref{eq:disc-action},
\[
 \int_{\TM}\big(L_G+c(G)\big)\,d\tilde\mu_x^{\lambda_n}
 =\lambda_n u_G^{\lambda_n}(x)\longrightarrow0,
\]
where the convergence holds by assumption.
The last assertion follows from Propositions~\ref{prop:compact}
and \ref{prop:pos-bdd} when $V$ has constant sign; along a locally
uniformly convergent sequence, it follows directly from the
boundedness of $u_G^{\lambda_n}(x)$.
\end{proof}

\section{Selection for nonpositive potentials}\label{sec:selection}

Throughout this section we assume $V\leqslant0$.
Corollary~\ref{cor:clusters} and Proposition~\ref{prop:cluster-norm}
already give the existence of accumulation points and their
critical equation, lower bound and Mather constraints.
We now identify all these accumulation points with the same function.

Let $\sol_-$ be the family of bounded subsolutions $w$ of
\eqref{eq:criticalG} satisfying
\begin{equation}\tag{$\mathrm{Av}_G$}\label{eq:Mather-norm}
 \int_{\MM}w(y)\,d\mu(y)\leqslant0
 \qquad\hbox{for every $\mu\in\Mis(G)$,}
\end{equation}
with the additional requirement
\begin{equation}\label{eq:upper-bound}
 w\leqslant u_H^0\quad\hbox{in $\MM$}
 \qquad\hbox{whenever $c(G)=c(H)$.}
\end{equation}
When $c(G)>c(H)$, only \eqref{eq:Mather-norm} is imposed.
A condition indexed by an empty family of measures is understood
to be vacuous. 

We start by remarking that  $\sol_-$ is nonempty. Indeed, by subtracting a further
constant from the function $v_G$ in
Proposition~\ref{prop:bdd-subsol}, we can arrange that
$v_G\leqslant\min\{0,u_H^0\}$, so that $v_G\in\sol_-$.

We can now state the main result in its precise form.

\begin{teorema}\label{thm:main}
Assume that $V\leqslant0$. Then $\sol_-$ is nonempty and the function
\begin{equation}\tag{$\mathrm{Sel}_{-}$}\label{eq:selected}
 u_G^0(x):=\sup_{w\in\sol_-}w(x),\qquad x\in\MM,
\end{equation}
is finite everywhere. Furthermore,
\[
 u_G^\lambda\ucv u_G^0\quad\hbox{in $\MM$ as $\lambda\to0^+$.}
\]
The limit $u_G^0$ is a Lipschitz solution of \eqref{eq:criticalG},
bounded from below and satisfying the Mather constraints
\eqref{eq:Mather-norm}. If $c(G)=c(H)$, it also satisfies
\eqref{eq:upper-bound}; in particular, $u_G^0$ is bounded
and belongs to $\sol_-$.
\end{teorema}

The boundedness requirement in the definition of $\sol_-$ concerns
each candidate separately and does not imply boundedness of their
supremum. When $c(G)>c(H)$, the selected solution is in fact coercive;
see Proposition~\ref{prop:gap}, which also gives an equivalent
characterization using possibly unbounded subsolutions.

\subsection{Proof of the selection theorem}\label{subsec:proof}

We first use the Mather constraints already established in
Proposition~\ref{prop:cluster-norm} and a truncation argument
to show that every accumulation point lies below the supremum
in \eqref{eq:selected}, which at this stage may take the value
$+\infty$.
We then use the decomposition of discounted measures to prove
the reverse inequality and establish that $u_G^0$ is finite
everywhere and is a Lipschitz critical subsolution.

\begin{prop}\label{prop:upper-sel}
For every $u\in\omegasetG$, one has $u\leqslant u_G^0$ on $\MM$.
\end{prop}

\begin{proof}
Fix $u\in\omegasetG$. By Corollary~\ref{cor:clusters}, $u$ is a
critical solution with $u\geqslant v_G$, and
Proposition~\ref{prop:cluster-norm} gives its Mather constraints.
However, $u$ need not be bounded, and therefore need not
itself belong to $\sol_-$. For $n\in\N$, define instead
\begin{equation*}
 w_n(x):=\min\{u(x),v_G(x)+n\},\qquad x\in\MM.
\end{equation*}
Both functions in the minimum are critical subsolutions.
Their minimum is a subsolution by convexity, and
$v_G\leqslant w_n\leqslant v_G+n$, so $w_n$ is bounded.
Since $w_n\leqslant u$, it satisfies \eqref{eq:Mather-norm}.
If the critical values agree, then
$w_n\leqslant u\leqslant u_H^0$, and it also satisfies
\eqref{eq:upper-bound}. Hence $w_n\in\sol_-$ for every $n$.

Moreover, $w_n=u$ on every fixed compact set for all
sufficiently large $n$, so $w_n\ucv u$ in $\MM$.
It follows that $u\leqslant u_G^0$ on $\MM$.
\end{proof}

To prove the reverse inequality, we first establish the estimate
needed to pass to the limit in the discounted integrals.

\begin{lemma}\label{lem:limsup}
Assume $V\leqslant0$, fix $x\in\MM$, and let $\lambda_n\to0^+$
be such that the discounted measures have the limits
\eqref{eq:mu-limits}. Then, for every $w\in\sol_-$,
\begin{equation*}
 \limsup_{n\to\infty}
 \int_{\MM}w(y)\,d\mu_x^{\lambda_n}(y)\leqslant0.
\end{equation*}
\end{lemma}

\begin{proof}
We use the notation of Proposition~\ref{prop:decomp} and write
$\mu:=\pi_{1\#}\tilde\mu$, $\nu:=\pi_{1\#}\tilde\nu$ and
$\nu_\infty:=\pi_{1\#}\tilde\nu_\infty$ for the projections onto
the respective base manifolds.

If $\tilde\nu_\infty=0$, then $\tilde\mu$ is a Mather probability
measure and $\tilde\mu_x^{\lambda_n}\weakcv\tilde\mu$ narrowly
in $\TM$. Hence $\mu\in\Mis(G)$ and
$\mu_x^{\lambda_n}\weakcv\mu$ narrowly in $\MM$.
Boundedness and continuity of $w$ give
\[
 \lim_n\int_{\MM}w\,d\mu_x^{\lambda_n}
 =\int_{\MM}w(y)\,d\mu(y)\leqslant0.
\]

Suppose next that $\tilde\nu_\infty\ne0$.
The action identity forces $c(G)=c(H)$, so $w\leqslant u_H^0$.
The function $u_H^0-w$ is nonnegative and continuous.
By vague convergence, the uniform velocity bound and an
increasing sequence of compact spatial cutoffs, we have
\[
 \int_{\MM}(u_H^0-w)(y)\,d\mu(y)
 \leqslant\liminf_n
 \int_{\MM}(u_H^0-w)(y)\,d\mu_x^{\lambda_n}(y).
\]
On the other hand, periodicity of $u_H^0$ gives
\[
 \lim_n\int_{\MM}u_H^0\,d\mu_x^{\lambda_n}
 =\int_{\TT^d}u_H^0(y)\,d\nu(y).
\]
Subtracting the two relations yields
\begin{equation}\label{eq:limsup-parts}
\begin{split}
 \limsup_n\int_{\MM}w\,d\mu_x^{\lambda_n}
 \leqslant{}&\int_{\MM}w(y)\,d\mu(y)\\
            &+\int_{\TT^d}u_H^0(y)\,d\nu_\infty(y).
\end{split}
\end{equation}
If $\mu=0$, the first integral vanishes; otherwise, its
nonpositivity follows from \eqref{eq:Mather-norm} applied to
$\mu/\mu(\MM)$, which belongs to $\Mis(G)$
by \eqref{eq:norm-limits}.
The second integral is nonpositive as well by
\eqref{eq:period-norm} and \eqref{eq:norm-limits}.
This proves the assertion.
\end{proof}

We can now prove the reverse inequality and show that
$u_G^0$ is finite everywhere and Lipschitz.

\begin{prop}\label{prop:lower-sel}
For every $u\in\omegasetG$ and $w\in\sol_-$, one has
$u\geqslant w$ on $\MM$.
In particular, $\sol_-$ is locally uniformly bounded from above,
and the function $u_G^0$ in \eqref{eq:selected} is finite
everywhere and is a Lipschitz critical subsolution.
\end{prop}

\begin{proof}
Fix $u\in\omegasetG$ and choose a sequence $\lambda_n\to0^+$
such that $u_G^{\lambda_n}\ucv u$ in $\MM$.

Let $x\in\MM$. The discounted measures
$\tilde\mu_x^{\lambda_n}$ defined by \eqref{eq:disc-measure}
are probability measures supported in
$\MM\times\overline B_{\kappa_*}$, where $\kappa_*$ is the
uniform speed bound in Proposition~\ref{prop:calibrated}.
Their masses are therefore uniformly bounded, and their
projections $\tilde\pi_\#\tilde\mu_x^{\lambda_n}$ are
probability measures supported in the compact set
$\TT^d\times\overline B_{\kappa_*}$.
By weak-* compactness of measures of bounded mass and
compactness of probability measures on this latter set,
we may pass to a common subsequence, still indexed by $n$,
such that
\[
\begin{aligned}
 \tilde\mu_x^{\lambda_n}
 &\weakst\tilde\mu
 &&\quad\hbox{vaguely in $\TM$,}\\
 \tilde\pi_\#\tilde\mu_x^{\lambda_n}
 &\weakcv\tilde\nu
 &&\quad\hbox{narrowly on $\TTorus$.}
\end{aligned}
\]
Here $\tilde\mu$ is a nonnegative finite measure with
$\tilde\mu(\TM)\leqslant1$, and $\tilde\nu$ is a probability
measure. By Lemma~\ref{lem:discounted-decomp}, the sequence
$(\tilde\mu_x^{\lambda_n})_n$ satisfies the hypotheses of
Proposition~\ref{prop:decomp}, so the conclusions of that
proposition apply to these limits.

Now let $w\in\sol_-$. Since $w$ is a bounded critical
subsolution, Proposition~\ref{prop:integrals}\,(i) gives
\[
 w(x)\leqslant u_G^{\lambda_n}(x)
       +\int_{\MM}w(y)\,d\mu_x^{\lambda_n}(y).
\]
Using $u_G^{\lambda_n}(x)\to u(x)$ and
Lemma~\ref{lem:limsup}, we obtain
\[
 w(x)\leqslant u(x)
       +\limsup_{n\to\infty}
          \int_{\MM}w(y)\,d\mu_x^{\lambda_n}(y)
       \leqslant u(x).
\]
Since $x$ and $w$ were arbitrary, this proves $w\leqslant u$
on $\MM$ for every $w\in\sol_-$.

By Corollary~\ref{cor:clusters}, the set $\omegasetG$
is nonempty. Fixing any $u\in\omegasetG$ and recalling that
$v_G\in\sol_-$, we therefore have
\[
 v_G(x)\leqslant u_G^0(x)
 :=\sup_{w\in\sol_-}w(x)\leqslant u(x)
 \qquad\hbox{for every $x\in\MM$.}
\]
Uniform coercivity of $G$ gives a common Lipschitz constant
$\kappa$ for its critical subsolutions. Since $u_G^0$ is finite
everywhere, it is also $\kappa$--Lipschitz. In particular,
the family $\sol_-$ is locally uniformly bounded from above.

Finally, stability of viscosity subsolutions under suprema
shows that $u_G^0$ is a critical subsolution.
This proves the last assertion.
\end{proof}

\begin{proof}[Proof of Theorem~\ref{thm:main}]
The nonemptiness of $\sol_-$ was established above, while
Proposition~\ref{prop:lower-sel} shows that $u_G^0$ is finite
everywhere and Lipschitz. By Proposition~\ref{prop:compact},
the family of discounted solutions is locally precompact.
Propositions~\ref{prop:upper-sel} and \ref{prop:lower-sel}
show that every accumulation point equals $u_G^0$.
Hence $u_G^\lambda\ucv u_G^0$ as $\lambda\to0^+$.

Since $u_G^0\in\omegasetG$, Corollary~\ref{cor:clusters}
shows that it solves \eqref{eq:criticalG} and is bounded from
below by $v_G$. Proposition~\ref{prop:cluster-norm} gives its
integrability against every $\mu\in\Mis(G)$ and the constraints
\eqref{eq:Mather-norm}.

If $c(G)=c(H)$, the same corollary gives
$v_G\leqslant u_G^0\leqslant u_H^0$. Thus $u_G^0$ is bounded,
satisfies \eqref{eq:upper-bound} and belongs to $\sol_-$.
\end{proof}

\subsection{Further properties of the selected solution and Mather measures}
\label{subsec:conseq}

The proof above applies without separating the critical-value regimes.
We now record further properties of the selected solution and
Mather measures, distinguishing between the positive-gap case
$c(G)>c(H)$ and the zero-gap case $c(G)=c(H)$.

\begin{prop}\label{prop:gap}
Assume $V\leqslant0$ and $c(G)>c(H)$. Then:
\begin{itemize}
\item[(i)] $\tilde\Mis(G)\ne\emptyset$, and there is a compact set
$K\subset\MM$ such that $\supp(\mu)\subset K$ for every
$\mu\in\Mis(G)$;
\smallskip
\item[(ii)] $c_f(G)=c(G)$;
\smallskip
\item[(iii)] formula \eqref{eq:selected} remains valid if the boundedness
requirement on $w$ is omitted;

\item[(iv)] there exist $\eps>0$ and $C>0$ such that
\[
 u_G^0(x)\geqslant\eps|x|-C\qquad\hbox{for every $x\in\MM$.}
\]
\end{itemize}
\end{prop}

\begin{proof}
Set $\delta:=c(G)-c(H)>0$.

\smallskip

{(i)}
Fix $x\in\MM$ and a sequence $\lambda_n\to0^+$.
By Lemma~\ref{lem:discounted-decomp}, the discounted measures
$\tilde\mu_x^{\lambda_n}$ satisfy the hypotheses of
Proposition~\ref{prop:decomp}. After extracting a subsequence,
let $\tilde\mu$ and $\tilde\nu_\infty$ be the measures provided
by that proposition. Since $\delta>0$, the action identity
\eqref{eq:act-balance} implies
\[
 \delta\,\tilde\nu_\infty(\TTorus)=0.
\]
Thus $\tilde\nu_\infty=0$ and $\tilde\mu(\TM)=1$.
By \eqref{eq:norm-limits}, $\tilde\mu\in\tilde\Mis(G)$,
which proves nonemptiness.

To prove the common compact support assertion, we first establish
a linear lower bound for
\[
 S(x):=S_{c(G)}^G(0,x),\qquad x\in\MM.
\]
By Proposition~\ref{prop:metric}, $S$ is a finite Lipschitz
critical subsolution. Applying Lemma~\ref{lem:smoothing} to
$u_H^0$, with periodic convolution, gives a smooth periodic
function $\psi$ such that
\[
 H(x,D\psi(x))\leqslant c(H)+\frac{\delta}{2}
 \qquad\hbox{for every $x\in\MM$.}
\]
Set $P:=\|D\psi\|_\infty$, and let $\ell\geqslant1$ be a
uniform Lipschitz constant for $H(x,\cdot)$ on
$\overline B_{P+1}$. Choose
\[
 \eps:=\min\left\{1,\frac{\delta}{2\ell}\right\}>0.
\]
Then, for every $x\in\MM$ and $|p|\leqslant\eps$,
\[
 H(x,D\psi(x)+p)
 \leqslant H(x,D\psi(x))+\ell|p|
 \leqslant c(H)+\frac{\delta}{2}+\ell\eps
 \leqslant c(G).
\]
Fenchel's inequality therefore yields
\[
 L_H(x,q)+c(G)
 \geqslant
 L_H(x,q)+ H(x,D\psi(x)+p)
\geqslant
\langle D\psi(x)+p,q\rangle
 \qquad\hbox{whenever $|p|\leqslant\eps$.}
\]
Taking the supremum over such $p$, we obtain
\begin{equation}\label{eq:strictmetric}
 L_H(x,q)+c(G)
 \geqslant\langle D\psi(x),q\rangle+\eps|q|
 \qquad\hbox{for every $(x,q)\in\TM$.}
\end{equation}

Choose $r>0$ such that $\supp(V)\subset B_r$, and let
$x\notin\overline B_r$. Fix $t>0$ and any Lipschitz curve
$\gamma:[-t,0]\to\MM$ with $\gamma(-t)=0$ and $\gamma(0)=x$.
Define
\[
 \tau:=\max\{s\in[-t,0]:\gamma(s)\in\overline B_r\},
 \qquad z:=\gamma(\tau).
\]
By continuity, $\tau\in(-t,0)$, $z\in\partial B_r$, and
$\gamma(s)\notin\overline B_r$ for every $s\in(\tau,0]$.
In particular, $V(\gamma(s))=0$ on this last portion.

Let $v_G$ be a bounded critical subsolution of
$G(x,Dv_G)=c(G)$, whose existence is guaranteed by
Proposition~\ref{prop:bdd-subsol}. By domination along
the first portion of $\gamma$, we have
\[
 \int_{-t}^{\tau}
 \big(L_G(\gamma(s),\dot\gamma(s))+c(G)\big)\,ds
 \geqslant v_G(z)-v_G(0)
 \geqslant-2\|v_G\|_\infty.
\]
On the remaining portion, $L_G=L_H$, so
\eqref{eq:strictmetric} gives
\begin{align*}
 \int_{\tau}^{0}
 \big(L_G(\gamma(s),\dot\gamma(s))+c(G)\big)\,ds
 &\geqslant\psi(x)-\psi(z)
          +\eps\int_\tau^0|\dot\gamma(s)|\,ds\\
 &\geqslant-2\|\psi\|_\infty+\eps|x-z|\\
 &\geqslant-2\|\psi\|_\infty+\eps(|x|-r).
\end{align*}
Adding the two estimates and taking the infimum over the curves
and travel times in \eqref{eq:actionS} yields
\[
 S(x)\geqslant
 \eps|x|-\eps r-2\|\psi\|_\infty-2\|v_G\|_\infty
 \qquad\hbox{if $|x|>r$.}
\]
For $|x|\leqslant r$, domination gives
$S(x)\geqslant v_G(x)-v_G(0)\geqslant-2\|v_G\|_\infty$.
Consequently,
\begin{equation}\label{eq:Scoercive}
 S(x)\geqslant\eps|x|-C_1
 \qquad\hbox{for every $x\in\MM$,}
\end{equation}
where we may take
\[
 C_1:=\eps r+2\|\psi\|_\infty+2\|v_G\|_\infty.
\]

We now use this estimate in the construction of
Proposition~\ref{prop:bdd-subsol}. Choose $k>0$ such that
$S<u_H^0+k$ on a neighborhood of $\supp(V)$, and set
\[
 \hat v_G:=\min\{S,u_H^0+k\}.
\]
Near $\supp(V)$, this function agrees with $S$.
Away from $\supp(V)$, both functions in the minimum are
subsolutions at level $c(G)$, since $G=H$ there and
$c(H)<c(G)$. Thus $\hat v_G$ is a critical subsolution.
It is bounded from above by $u_H^0+k$ and bounded from below
by \eqref{eq:Scoercive} and the boundedness of $u_H^0$.

Choose $R>r$ sufficiently large that
\[
 \eps R-C_1>\|u_H^0\|_\infty+k.
\]
Then \eqref{eq:Scoercive} implies
\[
 \hat v_G=u_H^0+k
 \qquad\hbox{on $\MM\setminus B_R$.}
\]
In particular,
\[
 G(x,D\hat v_G(x))\leqslant c(H)
 \qquad\hbox{for a.e. $x\in\MM\setminus B_R$.}
\]
Set $K:=\overline B_{R+1}$, and choose a smooth function
$\chi:\MM\to[0,1]$ such that $\chi=0$ on $B_R$ and
$\chi=1$ on $\MM\setminus K$. The preceding inequalities give
\[
 G(x,D\hat v_G(x))\leqslant c(G)-\delta\chi(x)
 \qquad\hbox{for a.e. $x\in\MM$.}
\]
Since the right-hand side is bounded and uniformly continuous,
Lemma~\ref{lem:action}, applied to $\hat v_G$, yields
\[
 0\leqslant\int_{\TM}\bigl(L_G+c(G)-\delta\chi\bigr)
       \,d\tilde\mu
 =-\delta\int_{\MM}\chi\,d\mu
 \qquad\hbox{for every $\tilde\mu\in\tilde\Mis(G)$,}
\]
where $\mu:=\pi_{1\#}\tilde\mu$. Thus $\int\chi\,d\mu=0$.
As $\chi=1$ on $\MM\setminus K$, it follows that
$\mu(\MM\setminus K)=0$, and hence $\supp(\mu)\subset K$.
The construction of $K$ is independent of $\tilde\mu$, so the
same compact set contains the supports of all measures in $\Mis(G)$.
This completes the proof of (i).\smallskip

{(ii)}
By (i), we can choose $\tilde\mu\in\tilde\Mis(G)$.
Applying Lemma~\ref{lem:action} to a subsolution at level $c_f(G)$ gives
\[
 0\leqslant\int_{\TM}\bigl(L_G+c_f(G)\bigr)\,d\tilde\mu
 =c_f(G)-c(G)\leqslant0.
\]
Thus $c_f(G)=c(G)$, proving (ii).\smallskip

(iii). 
Let $w$ be any critical subsolution, possibly unbounded,
satisfying \eqref{eq:Mather-norm}.
The integrals in that condition are finite since every
$\mu\in\Mis(G)$ is supported in $K$. For $n\in\N$, set
\[
 w_n:=\min\{\max\{w,v_G-n\},v_G+n\}.
\]
Stability of critical subsolutions under maxima and minima
shows that $w_n$ is a critical subsolution. Moreover,
\[
 v_G-n\leqslant w_n\leqslant v_G+n,
\]
so $w_n$ is bounded.  

For every bounded set $B\subset\R^d$, there exists
$n(B)\in\N$ such that
\[
 v_G-n\leqslant w\leqslant v_G+n
 \quad\hbox{on $B$,}\qquad\hbox{for every $n\geqslant n(B)$.}
\]
Thus $w_n=w$ on $B$ for every $n\geqslant n(B)$.
In particular, taking $B=K$ and recalling (i), we obtain
\[
 \int_{\MM}w_n\,d\mu
 =\int_{\MM}w\,d\mu\leqslant0
 \qquad
 \hbox{for every $\mu\in\Mis(G)$ and $n\geqslant n(K)$.}
\]
Since $c(G)>c(H)$, these are the only normalization
constraints in the definition of $\sol_-$.
Thus $w_n\in\sol_-$ for every $n\geqslant n(K)$.

The eventual equality on every bounded set also shows that
$w_n\ucv w$ in $\R^d$, i.e., any unbounded critical 
subsolution satisfying \eqref{eq:Mather-norm} can be obtained as a locally uniform 
limit of elements of $\sol_-$. 
Thus allowing unbounded critical subsolutions does not
increase the supremum in \eqref{eq:selected}, proving (iii).
\smallskip

{(iv).}
Since $S$ is continuous and $K$ is compact, the number
\[
 M_K:=\max_K S
\]
is finite. The function $S-M_K$ is a critical subsolution
and is nonpositive on $K$. By (i), it therefore satisfies
\[
 \int_{\MM}(S-M_K)\,d\mu\leqslant0
 \qquad\hbox{for every $\mu\in\Mis(G)$.}
\]
Although this subsolution is unbounded (in fact, it is coercive,
see \eqref{eq:Scoercive}), (iii) gives
$S-M_K\leqslant u_G^0$. Combining this inequality with
\eqref{eq:Scoercive}, we conclude that
\[
 u_G^0(x)\geqslant S(x)-M_K
 \geqslant\eps|x|-C_1-|M_K|
 \qquad\hbox{for every $x\in\MM$.}
\]
This proves (iv) with $C:=C_1+|M_K|$.
\end{proof}

\begin{esempio}\label{ex:coercive}
Let $d=1$, $H(x,p)=p^2/2$ and 
$V\in\CC_c^\infty(\R)$ with $-1\leqslant V\leqslant0$, and $V(x)=-1$
if and only if $x=0$. Then $c(H)=0$ and $c(G)=c_f(G)=1$.
Indeed, zero is a subsolution at level one, while at any smaller
level the almost-everywhere subsolution inequality is impossible near
the minimum of $V$. Since
\[
 L_G(x,q)+1=\tfrac12q^2+1+V(x),
\]
the unique Mather probability is $\delta_{(0,0)}$. The selection
formula gives
\[
 u_G^0(x)=\left|\int_0^x\sqrt{2(1+V(s))}\,ds\right|.
\]
In fact, every normalized critical subsolution is bounded from above
by this function, and truncating it from above by any positive
constant gives an admissible bounded subsolution. In particular,
$u_G^0(x)=\sqrt2\,|x|+O(1)$ at infinity. Thus the boundedness of each
$u_G^\lambda$ does not give a uniform global bound as $\lambda\to0$.
\end{esempio}

We now turn to the zero-gap case. The following proposition
shows that, under a strictness assumption near $\supp(V)$,
every bounded critical subsolution lying below $u_H^0$
automatically satisfies the Mather constraints. These can
therefore be omitted from the selection formula.

\begin{prop}\label{prop:strict}
Assume $V\leqslant0$ and $c(G)=c(H)$. Suppose there exists a critical
subsolution $v$ which is uniformly strict on an open neighborhood $U$
of $\supp(V)$, that is,
\[
 G(x,Dv(x))\leqslant c(G)-\delta
 \qquad\hbox{for a.e. $x\in U$}
\]
for some $\delta>0$. Then
\begin{equation}\label{eq:upperonly}
 u_G^0(x)=\sup\{w(x):w\hbox{ is a bounded critical subsolution of $G$},
                         \ w\leqslant u_H^0\}.
\end{equation}
The same formula holds with bounded critical solutions in place of
subsolutions.
\end{prop}

\begin{proof}
Choose $\chi\in\CC_c^\infty(U)$ with $0\leqslant\chi\leqslant1$
and $\chi=1$ on $\supp(V)$. The critical subsolution inequality
and the strictness assumption imply
\[
 G(x,Dv(x))\leqslant c(G)-\delta\chi(x)
 \qquad\hbox{for a.e. $x\in\MM$.}
\]
Fix $\tilde\mu\in\tilde\Mis(G)$ and write
$\mu:=\pi_{1\#}\tilde\mu$. Lemma~\ref{lem:action} applies directly
to $v$, whether or not it is bounded, and gives
\[
 0\leqslant\int_{\TM}\bigl(L_G+c(G)-\delta\chi\bigr)
       \,d\tilde\mu
 =-\delta\int_{\MM}\chi\,d\mu.
\]
Since $\chi$ is nonnegative and equals one on $\supp(V)$,
we obtain $\mu(\supp(V))=0$.

The measure $\tilde\nu:=\tilde\pi_\#\tilde\mu$ is a closed
probability measure on $\TTorus$: closedness follows by
testing $\tilde\mu$ against the periodic lifts of smooth
functions on $\TT^d$. Moreover, $V=0$ $\mu$-almost everywhere,
so \eqref{eq:LG}, periodicity of $L_H$ and $c(G)=c(H)$ give
\[
 \int_{\TTorus}(L_H+c(H))\,d\tilde\nu
 =\int_{\TM}(L_H+c(H))\,d\tilde\mu
 =\int_{\TM}(L_G+c(G))\,d\tilde\mu
 =0.
\]
Thus $\tilde\nu\in\tilde\Mis(H)$.

Now let $w$ be a bounded critical subsolution of $G$ with
$w\leqslant u_H^0$, and set $\nu:=\pi_{1\#}\tilde\nu$.
By periodicity of $u_H^0$ and \eqref{eq:period-norm},
\[
 \int_{\MM}w\,d\mu
 \leqslant\int_{\MM}u_H^0\,d\mu
 =\int_{\TT^d}u_H^0\,d\nu
 \leqslant0.
\]
Since $\tilde\mu\in\tilde\Mis(G)$ was arbitrary, $w$
satisfies all the Mather constraints and hence belongs
to $\sol_-$. Formula \eqref{eq:upperonly} therefore follows
from Theorem~\ref{thm:main}.

Finally, the same theorem shows that $u_G^0$ is itself a
bounded critical solution satisfying $u_G^0\leqslant u_H^0$.
Consequently, restricting the supremum in
\eqref{eq:upperonly} to bounded critical solutions does not
change its value.
\end{proof}
Theorem~\ref{thm:main} itself does not require the strictness assumption
of Proposition~\ref{prop:strict}. When that assumption is
unavailable, the contribution retained in the whole space is handled
by the Mather constraints, while the remaining contribution is handled
by $u_H^0$ through \eqref{eq:limsup-parts}.

\section{Selection for nonnegative potentials}\label{sec:positive}

It is natural to ask whether the selection principle
can be extended to a larger class of compactly supported potentials.
The simplest case to investigate next is that of nonnegative
potentials.
Throughout this section we will therefore assume $V\geqslant0$ and write
$c:=c(G)=c(H)$, using Proposition~\ref{prop:corder}.

Proposition~\ref{prop:pos-bdd} gives uniform bounds and local
precompactness of the discounted solutions. By
Corollary~\ref{cor:clusters}, every accumulation point is a
bounded critical solution lying above $u_H^0$.
This reverses the comparison obtained for nonpositive potentials
in the zero-gap case and suggests a characterization by an infimum.
The remaining question is whether all accumulation points coincide
and how to identify their common value.

We now introduce a candidate for the selected limit.
Let $\sol_+$ denote the family of bounded solutions $v$ of
$G(x,Dv)=c$ in $\MM$ satisfying $v\geqslant u_H^0$ on $\MM$,
and define
\begin{equation}\label{eq:pos-envelope}
 \underline u_G(x):=\inf_{v\in\sol_+}v(x),
 \qquad x\in\MM.
\end{equation}
Applying Theorem~\ref{thm:pos-existence} to $u_H^0$ shows
that $\sol_+$ has a least element. Thus the infimum in
\eqref{eq:pos-envelope} is attained and defines a bounded
critical solution satisfying
\[
 u_H^0\leqslant\underline u_G\leqslant u_H^0+C
 \qquad\hbox{in $\MM$.}
\]

By Corollary~\ref{cor:clusters}, every accumulation point
$u$ of the discounted solutions belongs to $\sol_+$ and
therefore satisfies
\begin{equation}\label{eq:pos-envelope-bound}
 \underline u_G\leqslant u\qquad\hbox{in $\MM$.}
\end{equation}
Proposition~\ref{prop:cluster-norm} gives the Mather constraints for $u$.
Since both functions are bounded, it follows that
\begin{equation}\label{eq:pos-envelope-Mather}
 \int_{\MM}\underline u_G\,d\mu
 \leqslant\int_{\MM}u\,d\mu\leqslant0
 \qquad\hbox{for every $\mu\in\Mis(G)$.}
\end{equation}
Thus $\underline u_G$ itself satisfies all the Mather constraints,
and imposing them on the candidates in \eqref{eq:pos-envelope}
does not change the infimum. This observation holds for every
nonnegative potential, independently of the extra hypothesis below.

To identify $\underline u_G$ as the selected limit, it remains
to prove the reverse inequality in \eqref{eq:pos-envelope-bound}.
The following theorem establishes it under an additional
condition on the averages of $u_H^0$ against periodic
Mather measures.

\begin{teorema}\label{thm:positive}
Assume that $V\geqslant0$ and that
\begin{equation}\tag{ZA}\label{eq:zeroaverages}
 \int_{\TT^d}u_H^0(y)\,d\nu(y)=0
 \qquad\hbox{for every $\nu\in\Mis(H)$.}
\end{equation}
Then $u_G^\lambda\ucv u_G^0$ in $\MM$ as $\lambda\to0^+$,
where
\begin{equation}\tag{$\mathrm{Sel}_{+}$}\label{eq:pos-selected}
 u_G^0(x)=\inf_{v\in\sol_+}v(x),\qquad x\in\MM.
\end{equation}
The limit is a bounded critical solution. The infimum is
unchanged if the candidates are additionally required to
satisfy \eqref{eq:Mather-norm}.
\end{teorema}

The proof relies on an estimate obtained by using calibrated
curves of a candidate $v\in\sol_+$ as competitors in the
discounted representation formula. The next lemma establishes
this estimate without assuming \eqref{eq:zeroaverages}.

\begin{lemma}\label{lem:competitor}
Let $V\geqslant0$, let $v\in\sol_+$, and fix $x\in\MM$.
Choose a backward calibrated curve $\gamma_x$ for $v$,
with $\gamma_x(0)=x$, and define the probability measure
$\tilde m_x^\lambda$ on $\TM$ by
\[
 \int_{\TM}f(y,q)\,d\tilde m_x^\lambda(y,q)
 :=
 \lambda\int_{-\infty}^0\e^{\lambda s}
 f(\gamma_x(s),\dot\gamma_x(s))\,ds
 \quad\hbox{for every $f\in\CC_b(\TM)$.}
\]
Let $m_x^\lambda:=\pi_{1\#}\tilde m_x^\lambda$ be its
projection onto $\MM$. Then
\begin{equation}\label{eq:competitor}
 u_G^\lambda(x)\leqslant v(x)
 -\int_{\MM}u_H^0(y)\,dm_x^\lambda(y).
\end{equation}
Moreover, the measures $\tilde\pi_\#\tilde m_x^\lambda$
are narrowly precompact on $\TTorus$, and every accumulation
point as $\lambda\to0^+$ belongs to $\tilde\Mis(H)$.
\end{lemma}

\begin{proof}
The critical calibration identity \eqref{eq:crit-calibr}
implies, almost everywhere in time,
\[
 L_G(\gamma_x(s),\dot\gamma_x(s))+c
 =\frac{d}{ds}v(\gamma_x(s)).
\]
Using $\gamma_x$ as a competitor in \eqref{eq:disc-formula}
and integrating by parts gives
\begin{equation*}
 u_G^\lambda(x)\leqslant v(x)
 -\int_{\MM}v(y)\,dm_x^\lambda(y).
\end{equation*}
The boundary term at $-\infty$ vanishes because $v$ is
bounded. Since $v\geqslant u_H^0$, this proves
\eqref{eq:competitor}.

The speed bound for critical calibrated curves implies
that the measures $\tilde m_x^\lambda$ have a common
bounded velocity support. Their projections under
$\tilde\pi$ are therefore supported in a common compact
subset of $\TTorus$ and are narrowly precompact.
Let $\lambda_n\to0^+$ be such that
\[
 \tilde\pi_\#\tilde m_x^{\lambda_n}
 \weakcv\tilde\nu
 \qquad\hbox{narrowly on $\TTorus$.}
\]
For every smooth periodic test function $\varphi$,
integration by parts, as in \eqref{eq:closeddefect}, gives
\[
 \int_{\TM}\langle D\varphi(y),q\rangle
 \,d\tilde m_x^\lambda(y,q)
 =
 \lambda\left(
 \varphi(x)-\int_{\MM}\varphi(y)\,dm_x^\lambda(y)
 \right).
\]
The right-hand side tends to zero as $\lambda\to0^+$.
Passing to the limit along $\lambda_n$ shows that
$\tilde\nu$ is a closed probability measure on $\TTorus$.

Calibration and integration by parts also give
\[
 \int_{\TM}(L_G+c)\,d\tilde m_x^\lambda
 =
 \lambda\left(
 v(x)-\int_{\MM}v(y)\,dm_x^\lambda(y)
 \right)
 \longrightarrow0,
\]
again by boundedness of $v$.
Since $L_H\leqslant L_G$, periodicity of $L_H$ and the
uniform velocity bound yield
\[
 \int_{\TTorus}(L_H+c)\,d\tilde\nu
 =
 \lim_{n\to\infty}
 \int_{\TM}(L_H+c)\,d\tilde m_x^{\lambda_n}
 \leqslant0.
\]
The opposite inequality holds for every closed probability
measure on $\TTorus$ by Lemma~\ref{lem:action}.
Thus $\tilde\nu\in\tilde\Mis(H)$.
\end{proof}

We are now in a position to prove the selection theorem.

\begin{proof}[Proof of Theorem~\ref{thm:positive}]
Let $u$ be an accumulation point and choose a sequence
$\lambda_n\to0^+$ such that $u_G^{\lambda_n}\ucv u$ in $\MM$.
By Corollary~\ref{cor:clusters}, $u\in\sol_+$, and hence
$u\geqslant\underline u_G$.

Fix $x\in\MM$ and $v\in\sol_+$.
By Lemma~\ref{lem:competitor}, after passing to a further
subsequence we may assume that
\[
 \tilde\pi_\#\tilde m_x^{\lambda_n}
 \weakcv\tilde\nu
 \qquad\hbox{narrowly on $\TTorus$,}
\]
where $\tilde\nu\in\tilde\Mis(H)$.
Let $\nu:=\pi_{1\#}\tilde\nu\in\Mis(H)$ be its projection.
Since $u_H^0$ is continuous and periodic, taking the limit
in \eqref{eq:competitor} and using \eqref{eq:zeroaverages}
gives
\[
 u(x)\leqslant v(x)
 -\int_{\TT^d}u_H^0(y)\,d\nu(y)
 =v(x).
\]
As $x$ and $v$ were arbitrary, we obtain
$u\leqslant\underline u_G$ on $\MM$.
Thus every accumulation point equals $\underline u_G$.
Local precompactness from Proposition~\ref{prop:pos-bdd}
therefore gives
\[
 u_G^\lambda\ucv\underline u_G=u_G^0
 \qquad\hbox{in $\MM$ as $\lambda\to0^+$.}
\]

The assertion concerning the additional Mather constraints
follows from \eqref{eq:pos-envelope-Mather} and the observation
preceding the theorem.
\end{proof}

Condition \eqref{eq:zeroaverages} is stronger than
\eqref{eq:period-norm}, which only gives nonpositive
averages of $u_H^0$ against periodic Mather measures.
Without this additional condition, the average in the last
estimate of the preceding proof may be strictly negative,
and that estimate no longer implies $u\leqslant v$.
In fact, the convergence of $u_G^\lambda$ to
$\underline u_G$ can and does fail.
Proposition~\ref{prop:failed-inf} in Section~\ref{sec:examples}
provides a one-dimensional example in which
$u_G^\lambda\ucv u_G^0$, but
\[
 \underline u_G(a)<u_G^0(a)
 \qquad\hbox{at some point $a\in\MM$.}
\]
In view of \eqref{eq:pos-envelope-Mather}, adding the Mather
constraints to the infimum cannot restore
\eqref{eq:pos-selected}.

\section{Counterexamples for nonnegative potentials}\label{sec:examples}

In this section, we examine two formulas for the candidate limit,
suggested by the preceding results and by the one-dimensional theory.
The first example shows that the characterization by a supremum
subject only to whole-space Mather constraints, valid in a regime
of the one-dimensional theory in \cite{CDD23}, need not extend to
higher dimensions. The second shows that the infimum
\eqref{eq:pos-envelope} need not be selected when
\eqref{eq:zeroaverages} fails. In both examples the discounted
solutions converge; it is the proposed characterization of their
limit that fails.

\subsection{Failure of the characterization by whole-space Mather measures}
\label{subsec:drift}

For one-dimensional Tonelli Hamiltonians, \cite[Theorem~2(III)]{CDD23}
shows that, when $c_f(G)=c(G)=c_f(H)=c(H)=:c$, the vanishing discount
limit is given by the following pointwise supremum:
\begin{equation}\label{eq:sup-characterization}
\sup\left\{
w(x):
\begin{array}{l}
w\ \text{bounded critical subsolution of }
G(x,Dw)=c\ \IN\ \MM\\[1mm]
\displaystyle\int_{\MM}w\,d\mu\leqslant0
\quad\text{for every }\mu\in\Mis(G)
\end{array}
\right\}.
\end{equation}
The example below shows that, in higher dimensions, the
vanishing discount limit need not coincide with this supremum.
Indeed, all four critical values $c(G)$, $c_f(G)$, $c(H)$
and $c_f(H)$ agree, yet no whole-space Mather probability exists,
and the supremum subject to its constraints is infinite.\medskip

Let $d=2$, and set
\begin{equation}\label{eq:driftH}
 b(x):=\big(\cos(2\pi x_2),\sin(2\pi x_2)\big),\qquad
 H(x,p):=\tfrac12|p|^2+\langle b(x),p\rangle.
\end{equation}
Choose
\begin{equation}\label{eq:driftV}
 V\in\CC_c^\infty\big(B_{1/8}((0,1/4))\big),\qquad
 V\geqslant0,\qquad V\not\equiv0,
\end{equation}
and let $G=H-V$. Both Hamiltonians are smooth and uniformly convex
in the momentum, and $H$ is $\Z^2$--periodic in $x$. Their
Lagrangians are
\begin{equation}\label{eq:driftL}
 L_H(x,q)=\tfrac12|q-b(x)|^2,\qquad
 L_G(x,q)=\tfrac12|q-b(x)|^2+V(x).
\end{equation}
In particular, both are nonnegative.

\begin{prop}\label{prop:drift}
For the Hamiltonians \eqref{eq:driftH}--\eqref{eq:driftV}, the following
properties hold:
\begin{itemize}
\item[(i)] $c_f(G)=c(G)=c_f(H)=c(H)=0$;
\smallskip
\item[(ii)] $\tilde\Mis(G)=\emptyset$. There is also no closed
probability measure in $\R^2\times\R^2$ with zero $L_H$--action.
For both Lagrangians the infimum of the action over whole-space
closed probabilities is zero, but it is not attained;
\smallskip
\item[(iii)] as $\lambda\to0^+$, the discounted solutions
$u_G^\lambda$ increase pointwise and converge locally uniformly
to a bounded nonzero critical solution $u_G^0\geqslant0$.
\end{itemize}
\end{prop}

\begin{proof}
We give the proofs of the three assertions separately.
\smallskip

(i). Since $H(x,0)=0$ and $G(x,0)=-V(x)\leqslant0$, the zero function is
a solution for $H$ and a subsolution for $G$ at level zero. Thus
$c_f(F)\leqslant c(F)\leqslant0$ for $F=H,G$.

To exclude a negative free critical value, fix $A>1$.
For $T>0$, let $\gamma_T:[0,2T+2]\to\R^2$ be the closed
piecewise affine curve that traverses, in the indicated order,
the four sides of the rectangle with vertices
\[
(A,0),\quad(A+T,0),\quad(A+T,1/2),\quad(A,1/2).
\]
The curve starts and ends at $(A,0)$, traversing each
horizontal side in time $T$ and each vertical side in time one,
 with constant speed on each side.

On the lower horizontal side the velocity is
$(1,0)=b(\gamma_T)$; on the upper one it is
$(-1,0)=b(\gamma_T)$. These portions have zero action. The whole
rectangle avoids $\supp(V)$, while on either vertical portion
$|\dot\gamma_T|=1/2$ and $|b|=1$. Hence
\begin{equation*}
\int_0^{2T+2}L_G(\gamma_T,\dot\gamma_T)\,ds
=\int_0^{2T+2}L_H(\gamma_T,\dot\gamma_T)\,ds
\leqslant\frac94.
\end{equation*}
If a global subsolution $w$ existed for $F$ at a level $a<0$,
domination along the closed curve $\gamma_T$ would give
\[
0=w(\gamma_T(2T+2))-w(\gamma_T(0))
\leqslant
\int_0^{2T+2}L_F(\gamma_T,\dot\gamma_T)\,ds
+a(2T+2)
\leqslant\frac94+a(2T+2).
\]
This is impossible for sufficiently large $T$.
Thus no global subsolution exists at a negative level,
which proves (i). Notice that no boundedness assumption on $w$ is needed here.

(ii).  Since $c_f(G)=c_f(H)=0$ by (i),
Theorem~\ref{thm:free-action} shows that the infimum of the
action over closed probability measures on $\R^2\times\R^2$ is zero
for both $L_G$ and $L_H$. However, neither infimum is attained. 

Indeed, suppose that a closed probability $\tilde\mu$ on
$\R^2\times\R^2$ has zero $L_G$--action. Nonnegativity in
\eqref{eq:driftL} implies
\begin{equation*}
 q=b(x),\qquad V(x)=0
 \qquad\hbox{for $\tilde\mu$--a.e. $(x,q)$.}
\end{equation*}
Use the bounded smooth function $\varphi(x)=-\cos(2\pi x_2)$ as a
test for closedness. We obtain
\[
 0=2\pi\int_{\R^2}\sin^2(2\pi x_2)\,d(\pi_{1\#}\tilde\mu)(x).
\]
Thus $\tilde\mu$ is concentrated on the countable union of the sets
\[
 E_k:=\{(x,q):x_2=k/2,\ q=((-1)^k,0)\},\qquad k\in\Z.
\]
For each $k$, take $\eta_k\in\CC_c^\infty(\R)$ with
$\eta_k(k/2)=1$ and
$\supp(\eta_k)\subset(k/2-1/8,k/2+1/8)$. The function
\[
 \varphi_k(x):=\eta_k(x_2)\arctan(x_1)
\]
is bounded and smooth with bounded derivative. 
On the union of the $E_j$, the second component of the
velocity vanishes. 
Closedness therefore gives
\[
0=\int_{\R^2\times\R^2}
\bigl\langle D\bigl(\eta_k(x_2)\arctan x_1\bigr),q\bigr\rangle
\,d\tilde\mu(x,q)
=(-1)^k\int_{E_k}\frac{1}{1+x_1^2}\,d\tilde\mu(x,q).
\]
The integrand is strictly positive on $E_k$, so $\tilde\mu(E_k)=0$.
This holds for every $k$, contradicting the fact that their countable
union carries a probability measure.

The same argument applies to $L_H$: zero action again implies
$q=b(x)$, and the condition involving $V$ was not used in the
closedness argument.
This proves (ii).\smallskip\noindent
%
%
%

(iii). 
The convergence argument is the same as in the proof of
\cite[Proposition~3.1]{DFIZ1}, where constants are critical
subsolutions. We recall it for the reader's convenience.

Since $c(G)=0$ and $L_G\geqslant0$, the representation formula
gives $u_G^\lambda\geqslant0$.
If $0<\lambda_1<\lambda_2$, then $u_G^{\lambda_2}$ is a
subsolution of the discounted equation with parameter
$\lambda_1$, since
\[
 \lambda_1u_G^{\lambda_2}+G(x,Du_G^{\lambda_2})
 =(\lambda_1-\lambda_2)u_G^{\lambda_2}\leqslant0.
\]
Comparison therefore yields
$u_G^{\lambda_2}\leqslant u_G^{\lambda_1}$.
By Proposition~\ref{prop:pos-bdd}, the family is uniformly
bounded and equi--Lipschitz for $0<\lambda\leqslant1$.
It thus converges pointwise and locally uniformly, as
$\lambda\to0^+$, to a bounded nonnegative function $u_G^0$.
Viscosity stability gives $G(x,Du_G^0)=0$.
The limit cannot be identically zero, since
$G(x,0)=-V(x)$ and $V\not\equiv0$.
This proves (iii).
\end{proof}

This example rules out the direct extension of the characterization
in \cite[Theorem~2(III)]{CDD23} to higher dimensions.
Since $\Mis(G)=\emptyset$, every bounded critical subsolution
satisfies the Mather constraints.
Every constant is such a subsolution, and hence 
the pointwise supremum in \eqref{eq:sup-characterization} is identically $+\infty$. 
The same is true if only bounded critical solutions are admitted,
since $u_G^0+C$ is a bounded critical solution for every constant $C$.

On the other hand, $u_H^\lambda\equiv u_H^0\equiv0$, so
\eqref{eq:zeroaverages} holds. Theorem~\ref{thm:positive}
therefore applies and identifies $u_G^0$ as the least bounded
critical solution above zero. 

\begin{oss}
The example extends to every $d\geqslant2$ by setting
\[
 b(x)=(\cos(2\pi x_2),\sin(2\pi x_2),0,\ldots,0)
\]
and choosing a nonzero nonnegative potential supported in a small ball
contained in $\{1/8<x_2<3/8\}$. All additional coordinates are kept
constant in the competitors. The closedness tests above still exclude
a whole-space minimizing probability, and all the conclusions remain
valid.
\end{oss}

\subsection{An asymmetric double pendulum and the infimum characterization}
\label{subsec:1d}

We now turn to the infimum in \eqref{eq:pos-envelope}.
An asymmetric double pendulum shows that \eqref{eq:zeroaverages}
can fail even for a smooth, uniformly convex Hamiltonian.
A nonnegative compactly supported perturbation then gives a
counterexample to \eqref{eq:pos-selected}. Convergence follows
from the one-dimensional result in \cite{CDD23}; we will compute
the selected limit and $\underline u_G$ at one equilibrium and
show that their values differ.

Let $W$ be a smooth nonnegative $1$--periodic
function whose zero set is
\[
 Z:=\tfrac12\Z.
\]
Consider the mechanical Hamiltonian
$\mathcal H(x,p):=\tfrac12p^2-W(x)$, and write
\begin{equation*}
 s(x):=\sqrt{2W(x)},\qquad
 A:=\int_0^{1/2} s(x)\,dx,\qquad
 B:=\int_{1/2}^1 s(x)\,dx.
\end{equation*}
We assume $A<B$. The zero-energy level consists of the two
branches $p=\pm s(x)$, which enclose two regions of areas
$2A$ and $2B$ over $[0,1/2]$ and $[1/2,1]$, respectively.

Denote by $c_{\mathcal H}(P)$ the critical value of the periodic
equation
\[
 \mathcal H(x,P+u')=c_{\mathcal H}(P).
\]
Seeking a $1$--periodic Lipschitz solution $u$ amounts to choosing
a $1$--periodic derivative $u'$ with zero mean. With the change
of unknown $v(x):=u(x)+Px$, the problem becomes
\[
 \mathcal H(x,v')=c_{\mathcal H}(P),\qquad
 \int_0^1v'(x)\,dx=P,
\]
with $v'$ extended periodically. Equivalently,
$v(x+1)=v(x)+P$. Thus $P$ is the signed area under the
graph of $v'$ over one period.

At level zero, a subsolution must satisfy
$-s(x)\leqslant v'(x)\leqslant s(x)$ almost everywhere, so
integration gives $|P|\leqslant A+B$.
Conversely, if $|P|\leqslant A+B$, the choice
$v'(x)=P s(x)/(A+B)$ has mean $P$ and defines a subsolution
at level zero. No negative level admits a subsolution, since
$W$ vanishes. Consequently,
\[
 \min_{P\in\R}c_{\mathcal H}(P)=0,
 \qquad
 \{P:c_{\mathcal H}(P)=0\}=[-(A+B),A+B].
\]

A solution at level zero satisfies
$v'(x)\in\{-s(x),s(x)\}$ almost everywhere.
The viscosity condition determines how these branches can be joined:
away from $Z$, a jump from $+s$ to $-s$ is admissible, whereas
a jump from $-s$ to $+s$ is not.
Indeed, at an upward corner a test function touching $v$ from
below can have slope strictly between $-s$ and $s$, where
$\mathcal H(x,\cdot)<0$, contradicting the supersolution inequality.
At points of $Z$, the two branches meet and can be joined continuously.

We now choose
\begin{equation}\label{eq:pendulum-P}
 2A<P<A+B,
\end{equation}
which is possible because $A<B$, and set
\begin{equation}\label{eq:pendulum-H}
 H(x,p):=\mathcal H(x,P+p)=\tfrac12(p+P)^2-W(x).
\end{equation}
The upper bound on $P$ places it in the interior of the minimum
interval of $c_{\mathcal H}$, so $c(H)=0$.
The lower bound has a different role: even following the upper
branch throughout the smaller region produces an increment
$A$ smaller than $P/2$. Hence every periodic critical
solution satisfies
\[
 u(1/2)-u(0)
 =\int_0^{1/2} v'(x)\,dx-P/2
 \leqslant A-P/2<0.
\]
This explains geometrically why the values at the two equilibria
cannot both be zero.

Choose the unique $\xi\in(1/2,1)$ such that
\[
 2\int_\xi^1s(x)\,dx=A+B-P.
\]
Let $p$ be the $1$--periodic function whose graph follows
the upper branch of the zero-energy level on $(0,\xi)$
and the lower branch on $(\xi,1)$, as shown in blue in
Figure~\ref{fig:pendulum}. Its jump at $\xi$ is downward,
and
\[
 \int_0^1p(x)\,dx
 =A+B-2\int_\xi^1s(x)\,dx=P.
\]
Set
\[
 v(x):=\int_0^x p(t)\,dt,
 \qquad
 u(x):=v(x)-Px.
\]
The downward jump is admissible in the viscosity sense,
so $u$ is a $1$--periodic critical solution of $H$.
Moreover,
\[
 u(0)=0,\qquad u(1/2)=A-P/2<0,
\]
hence $u\leqslant0$ on $Z$.

\begin{lemma}\label{lem:pendulum-maximal}
The function $u$ constructed above is the largest
$1$--periodic critical subsolution of $H$ that is
nonpositive on $Z$.
\end{lemma}

\begin{proof}
Let $w$ be any such subsolution. Then $w(0)=w(1)\leqslant0$.
On $(0,\xi)$, the subsolution inequality gives
$w'\leqslant p-P$ almost everywhere, so
\[
 w(x)\leqslant w(0)+\int_0^x(p(t)-P)\,dt
 \leqslant u(x)
 \qquad\hbox{for $x\in[0,\xi]$.}
\]
On $(\xi,1)$, it instead gives $w'\geqslant p-P$
almost everywhere. Since $u(1)=0$, we obtain
\[
 w(x)\leqslant w(1)-\int_x^1(p(t)-P)\,dt
 \leqslant u(x)
 \qquad\hbox{for $x\in[\xi,1]$.}
\]
Periodicity yields $w\leqslant u$ on $\R$.
\end{proof}

\begin{figure}[htbp]
\centering
\begin{tikzpicture}[>=stealth,line cap=round,font=\small]
 \pgfmathdeclarefunction{pendulumS}{1}{%
  \pgfmathparse{abs(sin(360*#1))*(1-0.5*sin(360*#1))}}
 \pgfmathdeclarefunction{pendulumJ}{1}{%
  \pgfmathparse{#1<=0.5 ?
   (1-cos(360*#1))/(2*pi)-#1/4+sin(720*#1)/(16*pi) :
   1/pi-1/8+(cos(360*#1)+1)/(2*pi)
    +(#1-0.5)/4-sin(720*#1)/(16*pi)}}
 \pgfmathsetmacro{\pendulumXi}{0.864619016974}
 \begin{scope}[x=5.4cm,y=1.08cm]
  \fill[gray!13]
   plot[domain=0:1,samples=121] (\x,{pendulumS(\x)})
   -- plot[domain=1:0,samples=121] (\x,{-pendulumS(\x)})
   -- cycle;
  \draw[->] (-0.045,0) -- (1.075,0) node[right] {$x$};
  \draw[->] (0,-1.72) -- (0,1.83) node[above] {$p$};
  \draw[gray] plot[domain=0:1,samples=121]
   (\x,{pendulumS(\x)});
  \draw[gray] plot[domain=0:1,samples=121]
   (\x,{-pendulumS(\x)});
  \draw[very thick,blue!65!black]
   plot[domain=0:0.5,samples=61] (\x,{pendulumS(\x)})
   plot[domain=0.5:\pendulumXi,samples=61]
    (\x,{pendulumS(\x)});
  \draw[very thick,blue!65!black]
   plot[domain=\pendulumXi:1,samples=41]
    (\x,{-pendulumS(\x)});
  \draw[->,densely dashed,blue!65!black]
   (\pendulumXi,{pendulumS(\pendulumXi)}) --
   (\pendulumXi,{-pendulumS(\pendulumXi)});
  \node[below left] at (0,0) {$0$};
  \node[below] at (0.5,-0.04) {$1/2$};
  \node[below] at (1,-0.04) {$1$};
  \node[above right,inner sep=2pt] at (\pendulumXi,0) {$\xi$};
  \node at (0.25,-0.23) {$2A$};
  \node at (0.70,-0.34) {$2B$};
  \node[above] at (0.72,1.52) {$p=s(x)$};
  \node[below] at (0.72,-1.52) {$p=-s(x)$};
  \node[blue!65!black,above] at (0.25,0.53) {$p(x)$};
 \end{scope}
 \begin{scope}[shift={(7.2cm,-1.60cm)},x=4.7cm,y=5.0cm]
  \draw[->] (-0.04,0) -- (1.08,0) node[right] {$x$};
  \draw[->] (0,-0.02) -- (0,0.73);
  \draw[densely dotted,gray] (0.5,0) -- (0.5,0.25);
  \draw[densely dotted,gray] (1,0) -- (1,{2/pi});
  \draw[thick,dashed] (0,0) -- (1,0.5);
  \draw[very thick,blue!65!black]
   plot[domain=0:0.5,samples=61] (\x,{pendulumJ(\x)})
   plot[domain=0.5:\pendulumXi,samples=61]
    (\x,{pendulumJ(\x)});
  \draw[very thick,blue!65!black]
   plot[domain=\pendulumXi:1,samples=41]
    (\x,{0.5+2/pi-pendulumJ(\x)});
  \draw[gray,densely dotted,thick]
   plot[domain=\pendulumXi:1,samples=41]
    (\x,{pendulumJ(\x)});
  \fill[blue!65!black] (0.5,{1/pi-1/8}) circle[radius=1.3pt];
  \fill (0.5,0.25) circle[radius=1.3pt];
  \fill[blue!65!black] (1,0.5) circle[radius=1.3pt];
  \fill[gray] (1,{2/pi}) circle[radius=1.3pt];
  \node[below left] at (0,0) {$0$};
  \node[below] at (0.5,0) {$1/2$};
  \node[below] at (1,0) {$1$};
  \node[below right,inner sep=2pt] at (0.5,{1/pi-1/8}) {$A$};
  \node[above left,inner sep=2pt] at (0.5,0.25) {$P/2$};
  \node[right,inner sep=3pt] at (1,{2/pi}) {$A+B$};
  \node[right,inner sep=3pt] at (1,0.5) {$P$};
  \node[blue!65!black,above left] at (0.73,{pendulumJ(0.73)}) {$v$};
  \node[below right] at (0.77,0.385) {$Px$};
  \node[gray,above] at (0.91,0.63) {$J$};
 \end{scope}
\end{tikzpicture}
\caption{The data of Remark~\ref{rem:pendulum-data}.
Left: the zero-energy level and the function $p=v'$
(blue), with a downward jump at $\xi$ and mean $P$.
The two enclosed regions have areas $2A<2B$.
Right: its primitive $v$, normalized by $v(0)=0$, and the affine
function $Px$. They agree at $x=0,1$, while $v(1/2)=A<P/2$.
The dotted continuation is $J(x):=\int_0^x s(t)\,dt$, with
$J(1)=A+B>P$.}
\label{fig:pendulum}
\end{figure}
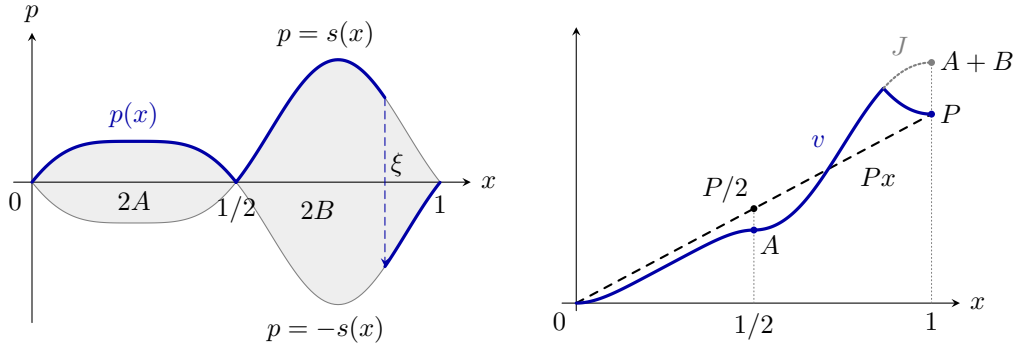

We next increase the action across the smaller region by a
compactly supported perturbation. Choose
$V\in\CC_c^\infty((0,1/2))$, $V\geqslant0$, such that
\begin{equation}\label{eq:pendulum-barrier}
 \widehat A:=\int_0^{1/2}\sqrt{2(W(x)+V(x))}\,dx>P/2.
\end{equation}
Set
\begin{equation}\label{eq:pendulum-G}
 \widehat W:=W+V,\qquad
 G(x,p):=H(x,p)-V(x)=\tfrac12(p+P)^2-\widehat W(x).
\end{equation}
The zero set of $\widehat W$ is still $Z$.

\begin{lemma}\label{lem:pendulum-measures}
The Hamiltonians \eqref{eq:pendulum-H} and \eqref{eq:pendulum-G}
satisfy
\[
 c_f(G)=c(G)=c_f(H)=c(H)=0.
\]
For either Hamiltonian, the whole-space Mather probabilities are
exactly the probabilities supported on $Z\times\{0\}$.
In particular, the whole-space constraints
\eqref{eq:Mather-norm} are equivalent to
\begin{equation}\label{eq:pendulum-norm}
 w(z)\leqslant0\qquad\hbox{for every $z\in Z$.}
\end{equation}
Finally, the periodic Mather measures for $H$ are exactly the
convex combinations of $\delta_{(0,0)}$ and $\delta_{(1/2,0)}$
on $\TT^1\times\R$.
\end{lemma}

\begin{proof}
We have $c(H)=0$ by construction, and $c(G)=0$ by
Proposition~\ref{prop:corder}.
Let $\rho:=P/(A+B)\in(0,1)$ and choose the $1$--periodic
$\CC^1$ function $\phi$ with $\phi(0)=0$ and
$\phi'(x)=\rho s(x)-P$.
For $F(x,p)=\tfrac12(p+P)^2-Q(x)$, with $Q=W$ or
$Q=\widehat W$, the Lagrangian is
\[
 L_F(x,q)=\tfrac12q^2-Pq+Q(x).
\]
Recalling that $s^2(x)=2W(x)$, we have 
\begin{equation*}
 L_F(x,q)-\phi'(x)q
 =\tfrac12(q-\rho s(x))^2+(1-\rho^2)W(x)+(Q(x)-W(x)).
\end{equation*}
The right-hand side is nonnegative and vanishes exactly on
$Z\times\{0\}$. Since closedness annihilates the integral of
$\phi'(x)q$, every whole-space closed probability has
nonnegative action, with equality if and only if it is
concentrated on that set. Conversely, every probability
supported there is closed and has zero action.
The infimum of the action is therefore zero, and
Theorem~\ref{thm:free-action} gives $c_f(H)=c_f(G)=0$.
For $Q=W$, the same identity on the torus characterizes the
periodic Mather measures. The description of the whole-space measures
also gives \eqref{eq:pendulum-norm}.
\end{proof}

By \cite[Proposition~1.3]{DFIZ1}, the periodic discounted limit
$u_H^0$ is the largest $1$--periodic critical subsolution whose
averages against all periodic Mather measures are nonpositive.
By Lemma~\ref{lem:pendulum-measures}, these constraints are
equivalent to being nonpositive on $Z$.
Lemma~\ref{lem:pendulum-maximal} therefore gives $u_H^0=u$.
In particular,
\begin{equation}\label{eq:pendulum-Hvalues}
 u_H^0(0)=0,\qquad
 u_H^0(1/2)=A-P/2<0.
\end{equation}
Thus the periodic Mather measure $\delta_{(1/2,0)}$ has
strictly negative average of $u_H^0$, so
\eqref{eq:zeroaverages} fails.

We now turn to the perturbed problem.
For $Q=W$ or $Q=\widehat W$, the critical semidistance of
$F(x,p)=\tfrac12(p+P)^2-Q(x)$ is
\begin{equation*}
 S_Q(y,x):=\left|\int_y^x\sqrt{2Q(t)}\,dt\right|-P(x-y).
\end{equation*}
Indeed, a Lipschitz function $w$ is a critical subsolution if and only if
\[
 -P-\sqrt{2Q(x)}\leqslant w'(x)\leqslant-P+\sqrt{2Q(x)}
 \qquad\hbox{for a.e. $x\in\R$.}
\]
Integration gives $w(x)-w(y)\leqslant S_Q(y,x)$.
For fixed $y$ and $x$, equality is attained by a primitive of
$-P+\sqrt{2Q}$ if $x\geqslant y$, and of $-P-\sqrt{2Q}$ if
$x\leqslant y$. The formula therefore follows from the maximality
property in Proposition~\ref{prop:metric}.

By Lemma~\ref{lem:pendulum-measures}, all the critical values
coincide with the free critical values and are zero.
Moreover, since $H$ is Tonelli, condition (U) in
\cite{CDD23} holds by Proposition~6.4 therein.
We may therefore apply \cite[Theorems~6.7 and 6.10]{CDD23}
to conclude that the discounted solutions $u_G^\lambda$
converge locally uniformly to
\begin{equation}\label{eq:pendulum-limits}
 u_G^0(x)=\inf_{z\in Z}S_{\widehat W}(z,x).
\end{equation}
For the perturbed problem, sources to the left and to the
right of $1/2$ satisfy, respectively,
\[
 \begin{aligned}
 S_{\widehat W}(z,1/2)
 =S_W(z,0)+\widehat A-P/2
 &\geqslant u_H^0(0)-u_H^0(z)+\widehat A-P/2\\
 &\geqslant\widehat A-P/2
 \qquad
 \FORALL z\in Z\ \hbox{with}\  z\leqslant 0,\\
 S_{\widehat W}(z,1/2)
 =S_W(z,1)+B+P/2
 &\geqslant u_H^0(1)-u_H^0(z)+B+P/2\\
 &\geqslant B+P/2
 \qquad
 \FORALL z\in Z\ \hbox{with}\  z\geqslant 1,
 \end{aligned}
\]
where the inequalities follow from domination for $u_H^0$,
since $u_H^0(0)=u_H^0(1)=0$ and $u_H^0\leqslant0$ on $Z$.
Both lower bounds are positive,
whereas the source $1/2$ has zero cost. Hence
\begin{equation}\label{eq:pendulum-Gvalue}
 u_G^0(1/2)=0.
\end{equation}

It remains to compare this value with $\underline u_G(1/2)$.
The semigroup construction in Theorem~\ref{thm:pos-existence}
makes this comparison immediate.

\begin{prop}\label{prop:failed-inf}
For the Hamiltonians \eqref{eq:pendulum-H} and
\eqref{eq:pendulum-G}, one has
\begin{equation}\label{eq:failed-inf}
 \underline u_G(1/2)=A-P/2<0=u_G^0(1/2).
\end{equation}
Thus \eqref{eq:pos-selected} fails, even if the candidates
in \eqref{eq:pos-envelope} are additionally required to
satisfy all the whole-space Mather constraints.
\end{prop}

\begin{proof}
The proof of Theorem~\ref{thm:pos-existence} gives
\[
 \underline u_G=\lim_{t\to+\infty}\S_G(t)u_H^0.
\]
For every $z\in Z$, we have $L_G(z,0)=0$. Using the constant
curve at $z$ in the definition of $\S_G(t)$, together with
\eqref{eq:pos-LO-monotone}, yields
\[
 u_H^0(z)\leqslant\bigl(\S_G(t)u_H^0\bigr)(z)
 \leqslant u_H^0(z)+tL_G(z,0)=u_H^0(z).
\]
Passing to the limit gives $\underline u_G(z)=u_H^0(z)$
for every $z\in Z$. Taking $z=1/2$ and using
\eqref{eq:pendulum-Hvalues} and \eqref{eq:pendulum-Gvalue}
proves \eqref{eq:failed-inf}.
Finally, by \eqref{eq:pos-envelope-Mather}, adding the Mather
constraints to the candidates does not change the infimum.
\end{proof}

\begin{oss}\label{rem:pendulum-data}
A concrete choice of the periodic data is
\[
 W(x)=\tfrac12\sin^2(2\pi x)
             \bigl(1-\tfrac12\sin(2\pi x)\bigr)^2,
 \qquad P=\tfrac12.
\]
In this case $A=1/\pi-1/8$, $B=1/\pi+1/8$, so
\eqref{eq:pendulum-P} holds and
$\underline u_G(1/2)=1/\pi-3/8<0=u_G^0(1/2)$
for every perturbation chosen as in \eqref{eq:pendulum-barrier}.
\end{oss}

These examples exhibit two distinct difficulties.
In the first, whole-space Mather probabilities are absent even
though all the critical values agree. The Mather constraints
are therefore vacuous, and the corresponding supremum is
identically $+\infty$.
In the second, such measures exist, but the least bounded
critical solution above $u_H^0$ differs from the selected
limit, despite satisfying all the Mather constraints.

The first example suggests that a general selection principle
for nonnegative potentials must also account for information
carried by minimizing curves that escape to infinity.
Such information already enters the one-dimensional characterization
when $c(G)=c(H)>c_f(H)$; see \cite[Theorem~2(II)]{CDD23}.
In that regime, the admissible subsolutions satisfy an additional
comparison with $u_H^0$ on one exterior half-line, chosen according
to the direction of the minimizing dynamics.
In higher dimensions, the periodic measures associated with the
escaping component retain information that the whole-space
Mather measures alone cannot record.

Agreement with $u_H^0$ outside a compact set is not, by itself,
a sufficient selection condition.
In the double-pendulum example, $u_G^0=u_H^0$ outside a
compact interval. Indeed, the construction of $u_H^0=u$ gives,
for every $x$, an integer $z$ with $|z-x|\leqslant1$ such that
$u_H^0(x)=S_W(z,x)$. For $|x|$ large, the interval joining
$z$ to $x$ avoids $\supp(V)$, so
$S_W(z,x)=S_{\widehat W}(z,x)$.
Formula~\eqref{eq:pendulum-limits} then gives
$u_G^0(x)\leqslant u_H^0(x)$, while comparison gives the
reverse inequality.
Since $u_H^0\leqslant\underline u_G\leqslant u_G^0$ by
\eqref{eq:pos-envelope-bound}, we also have
$\underline u_G=u_H^0$ outside the same compact interval.
Thus $u_G^0$ and $\underline u_G$ both satisfy the Mather
constraints and agree with $u_H^0$ outside a compact interval,
yet they differ at $1/2$. These requirements therefore do
not suffice to identify the selected limit.

Both examples have convergent discounted solutions; the failure
concerns the proposed characterization of their limit.
The results of this paper do not settle the general
multidimensional convergence problem for $V\geqslant0$
under the standing assumptions.
Theorem~\ref{thm:positive} requires the additional condition
\eqref{eq:zeroaverages}, which fails in the double-pendulum
example.

\section*{Acknowledgements}

The author began working on this problem in October 2017.
His earlier work had settled most of the case of nonpositive
potentials, but a substantial part of the equal-critical-value
regime remained unresolved.
The strategy of separating retained and escaping mass was
already part of this approach and was later implemented
in the one-dimensional work \cite[Section~3.2]{CDD23}.
Its implementation in higher dimensions, however,
remained a major difficulty.

The author gratefully acknowledges discussions on this problem
with Gonzalo Contreras, Albert Fathi and Hitoshi Ishii during
his stay at the Mathematical Sciences Research Institute in
Berkeley, California, from September to December 2018
(supported by the National Science Foundation under
Grant No.~DMS-1440140).
He also benefited from further discussions with Albert Fathi
and Antonio Siconolfi during Fathi's visits to Sapienza
University of Rome in May 2022 and May 2023
(supported through the 2021 and 2022 Sapienza Visiting
Professor calls, respectively), and in February 2026
(supported by the Dipartimento di Eccellenza project,
CUP~B83C23001390001).

On 24 September 2026, the author presented the remaining case
to ChatGPT (OpenAI, GPT-5 Astra, Medium reasoning setting),
which provided an argument resolving it.
He subsequently suggested seeking a common treatment of
$c(G)>c(H)$ and $c(G)=c(H)$ and asked ChatGPT
(GPT-6 Astra, Max reasoning setting) to develop his unpublished
draft into an article along these lines.
The resulting draft included a unified selection proof and
a common compactness argument; the author subsequently
revised and developed these arguments.

The specific contribution arising from these exchanges with
ChatGPT was the implementation of the strategy of separating
retained and escaping mass presented in
Proposition~\ref{prop:decomp}.

The compactness conclusions of Proposition~\ref{prop:compact}
had already been established by the author through separate
arguments in the two critical-value regimes.
The revised proof provides a common argument and an explicit
local upper bound, simplifying the earlier analysis.

The results for nonnegative potentials in
Theorem~\ref{thm:pos-existence}, Proposition~\ref{prop:pos-bdd}
and Section~\ref{sec:positive} are due to the author.
The counterexample to the supremum characterization in
Section~\ref{subsec:drift} was proposed by ChatGPT.
The counterexample to the infimum characterization in
Section~\ref{subsec:1d} builds on a periodic double-pendulum
construction suggested by Maxime Zavidovique, whom the
author warmly thanks.
The compactly supported perturbation yielding this
counterexample was developed in subsequent exchanges
with ChatGPT.

The author takes full responsibility for the mathematical
content of this paper and for any errors or omissions.

\bibliographystyle{siam}
\bibliography{discount}

@book {bardi,
    AUTHOR = {Bardi, Martino and Capuzzo-Dolcetta, Italo},
     TITLE = {Optimal control and viscosity solutions of
              {H}amilton-{J}acobi-{B}ellman equations},
    SERIES = {Systems \& Control: Foundations \& Applications},
      NOTE = {With appendices by Maurizio Falcone and Pierpaolo Soravia},
 PUBLISHER = {Birkh\"auser Boston Inc.},
   ADDRESS = {Boston, MA},
      YEAR = {1997},
     PAGES = {xviii+570},
      ISBN = {0-8176-3640-4},
   MRCLASS = {49-02 (49J15 49K15 49L25)},
  MRNUMBER = {1484411 (99e:49001)},
MRREVIEWER = {Vladimir Veliov},
       DOI = {10.1007/978-0-8176-4755-1},
       URL = {http://dx.doi.org/10.1007/978-0-8176-4755-1},
}

@ARTICLE{BK22,
       author = {{Bardi}, Martino and {Kouhkouh}, Hicham},
        title = "{An Eikonal equation with vanishing Lagrangian arising in Global Optimization}",
      journal = {arXiv e-prints},
         year = 2022,
        month = feb,
          eid = {arXiv:2202.02561},
        pages = {arXiv:2202.02561},
archivePrefix = {arXiv},
       eprint = {2202.02561},
 primaryClass = {math.OC},
       adsurl = {https://ui.adsabs.harvard.edu/abs/2022arXiv220202561B}
}

@book {barles,
    AUTHOR = {Barles, Guy},
     TITLE = {Solutions de viscosit\'e des \'equations de
              {H}amilton-{J}acobi},
    SERIES = {Math\'ematiques \& Applications (Berlin) [Mathematics \&
              Applications]},
    VOLUME = {17},
 PUBLISHER = {Springer-Verlag},
   ADDRESS = {Paris},
      YEAR = {1994},
     PAGES = {x+194},
      ISBN = {3-540-58422-6},
   MRCLASS = {49L25 (35D05 35F20 35J60)},
  MRNUMBER = {1613876 (2000b:49054)},
MRREVIEWER = {Martino Bardi},
}

@article {CDD23,
    AUTHOR = {Capuzzo Dolcetta, Italo and Davini, Andrea},
     TITLE = {On the vanishing discount approximation for compactly supported
              perturbations of periodic {H}amiltonians: the 1d case},
   JOURNAL = {Comm. Partial Differential Equations},
  FJOURNAL = {Communications in Partial Differential Equations},
    VOLUME = {48},
      YEAR = {2023},
    NUMBER = {4},
     PAGES = {576--622},
       DOI = {10.1080/03605302.2023.2183409},
       URL = {https://doi.org/10.1080/03605302.2023.2183409},
}

@article {CaPo19,
    AUTHOR = {Cardaliaguet, Pierre and Porretta, Alessio},
     TITLE = {Long time behavior of the master equation in mean field game
              theory},
   JOURNAL = {Anal. PDE},
  FJOURNAL = {Analysis \& PDE},
    VOLUME = {12},
      YEAR = {2019},
    NUMBER = {6},
     PAGES = {1397--1453},
      ISSN = {2157-5045},
   MRCLASS = {35Q91 (35B40 35F21 35G50 37N40 91A13)},
  MRNUMBER = {3921309},
MRREVIEWER = {Abderrahmane Habbal},
       DOI = {10.2140/apde.2019.12.1397},
       URL = {https://doi.org/10.2140/apde.2019.12.1397},
}

@article{Chen21,
    AUTHOR = {Chen, Qinbo},
     TITLE = {Convergence of solutions of {H}amilton--{J}acobi equations
              depending nonlinearly on the unknown function},
   JOURNAL = {Adv. Calc. Var.},
  FJOURNAL = {Advances in Calculus of Variations},
    VOLUME = {16},
      YEAR = {2023},
    NUMBER = {1},
     PAGES = {45--68},
      ISSN = {1864-8258},
       DOI = {10.1515/acv-2020-0089},
       URL = {https://doi.org/10.1515/acv-2020-0089},
}

@article {CCIsZh19,
    AUTHOR = {Chen, Qinbo and Cheng, Wei and Ishii, Hitoshi and Zhao, Kai},
     TITLE = {Vanishing contact structure problem and convergence of the
              viscosity solutions},
   JOURNAL = {Comm. Partial Differential Equations},
  FJOURNAL = {Communications in Partial Differential Equations},
    VOLUME = {44},
      YEAR = {2019},
    NUMBER = {9},
     PAGES = {801--836},
      ISSN = {0360-5302},
   MRCLASS = {35F21 (35D40 37J50 49L25)},
  MRNUMBER = {3952779},
MRREVIEWER = {Piotr Rybka},
       DOI = {10.1080/03605302.2019.1608561},
       URL = {https://doi.org/10.1080/03605302.2019.1608561},
}

@book {ConItu,
    AUTHOR = {Contreras, Gonzalo and Iturriaga, Renato},
     TITLE = {Global minimizers of autonomous {L}agrangians},
    SERIES = {22$^{\rm o}$ Col\'oquio Brasileiro de Matem\'atica. [22nd
              Brazilian Mathematics Colloquium]},
 PUBLISHER = {Instituto de Matem\'atica Pura e Aplicada (IMPA), Rio de
              Janeiro},
      YEAR = {1999},
     PAGES = {148},
      ISBN = {85-244-0151-6},
   MRCLASS = {37J50 (37J45 53D25)},
  MRNUMBER = {1720372 (2001j:37113)},
MRREVIEWER = {Victor Bangert},
}

@article {DaLi21,
    AUTHOR = {Davini, Andrea and Wang, Lin},
     TITLE = {On the vanishing discount problem from the negative direction},
   JOURNAL = {Discrete Contin. Dyn. Syst.},
  FJOURNAL = {Discrete and Continuous Dynamical Systems. Series A},
    VOLUME = {41},
      YEAR = {2021},
    NUMBER = {5},
     PAGES = {2377--2389},
      ISSN = {1078-0947},
   MRCLASS = {35F21 (37J51)},
  MRNUMBER = {4225917},
       DOI = {10.3934/dcds.2020368},
       URL = {https://doi.org/10.3934/dcds.2020368},
}

@article {DFIZ1,
    AUTHOR = {Davini, Andrea and Fathi, Albert and Iturriaga, Renato and
              Zavidovique, Maxime},
     TITLE = {Convergence of the solutions of the discounted
              {H}amilton-{J}acobi equation: convergence of the discounted
              solutions},
   JOURNAL = {Invent. Math.},
  FJOURNAL = {Inventiones Mathematicae},
    VOLUME = {206},
      YEAR = {2016},
    NUMBER = {1},
     PAGES = {29--55},
      ISSN = {0020-9910},
   MRCLASS = {35B27 (35B25 35D40 35F21 37J50 49L25)},
  MRNUMBER = {3556524},
MRREVIEWER = {Adina Ciomaga},
       DOI = {10.1007/s00222-016-0648-6},
       URL = {https://doi.org/10.1007/s00222-016-0648-6},
}

@article {DS,
    AUTHOR = {Davini, Andrea and Siconolfi, Antonio},
     TITLE = {A generalized dynamical approach to the large time behavior of
              solutions of {H}amilton-{J}acobi equations},
   JOURNAL = {SIAM J. Math. Anal.},
  FJOURNAL = {SIAM Journal on Mathematical Analysis},
    VOLUME = {38},
      YEAR = {2006},
    NUMBER = {2},
     PAGES = {478--502 (electronic)},
      ISSN = {0036-1410},
   MRCLASS = {49L25 (35B40 35F20 49K40)},
  MRNUMBER = {2237158 (2007d:49047)},
MRREVIEWER = {Fabio Bagagiolo},
       DOI = {10.1137/050621955},
       URL = {http://dx.doi.org/10.1137/050621955},
}

@article {DZ21,
    AUTHOR = {Davini, Andrea and Zavidovique, Maxime},
     TITLE = {Convergence of the solutions of discounted {H}amilton-{J}acobi
              systems},
   JOURNAL = {Adv. Calc. Var.},
  FJOURNAL = {Advances in Calculus of Variations},
    VOLUME = {14},
      YEAR = {2021},
    NUMBER = {2},
     PAGES = {193--206},
      ISSN = {1864-8258},
   MRCLASS = {35B40 (37J51 49L25)},
  MRNUMBER = {4236200},
       DOI = {10.1515/acv-2018-0037},
       URL = {https://doi.org/10.1515/acv-2018-0037},
}

@article {DZ10,
    AUTHOR = {Davini, Andrea and Zavidovique, Maxime},
     TITLE = {Weak {KAM} theory for nonregular commuting {H}amiltonians},
   JOURNAL = {Discrete Contin. Dyn. Syst. Ser. B},
  FJOURNAL = {Discrete and Continuous Dynamical Systems. Series B. A Journal
              Bridging Mathematics and Sciences},
    VOLUME = {18},
      YEAR = {2013},
    NUMBER = {1},
     PAGES = {57--94},
      ISSN = {1531-3492},
   MRCLASS = {37J40},
  MRNUMBER = {2982007},
       DOI = {10.3934/dcdsb.2013.18.57},
       URL = {http://dx.doi.org/10.3934/dcdsb.2013.18.57},
}

@article{II09,
    AUTHOR = {Ichihara, Naoyuki and Ishii, Hitoshi},
     TITLE = {Long-time behavior of solutions of {Hamilton--Jacobi}
              equations with convex and coercive {Hamiltonians}},
   JOURNAL = {Arch. Ration. Mech. Anal.},
    VOLUME = {194},
    NUMBER = {2},
      YEAR = {2009},
     PAGES = {383--419},
       DOI = {10.1007/s00205-008-0170-0},
}

@incollection{Fat24,
    AUTHOR = {Fathi, Albert},
     TITLE = {Viscosity solutions of the {Hamilton--Jacobi} equation
              on a noncompact manifold},
 BOOKTITLE = {Hamiltonian Systems: Dynamics, Analysis, Applications},
    SERIES = {Mathematical Sciences Research Institute Publications},
    VOLUME = {72},
 PUBLISHER = {Cambridge University Press},
      YEAR = {2024},
     PAGES = {111--172},
}

@article {fatmad,
    AUTHOR = {Fathi, Albert and Maderna, Ezequiel},
     TITLE = {Weak {KAM} theorem on non compact manifolds},
   JOURNAL = {NoDEA Nonlinear Differential Equations Appl.},
  FJOURNAL = {NoDEA. Nonlinear Differential Equations and Applications},
    VOLUME = {14},
      YEAR = {2007},
    NUMBER = {1-2},
     PAGES = {1--27},
      ISSN = {1021-9722},
     CODEN = {NDEAFL},
   MRCLASS = {49L25 (37J15 37J40)},
  MRNUMBER = {2346451 (2009h:49050)},
       DOI = {10.1007/s00030-007-2047-6},
       URL = {http://dx.doi.org/10.1007/s00030-007-2047-6},
}

@article {FS05,
    AUTHOR = {Fathi, Albert and Siconolfi, Antonio},
     TITLE = {P{DE} aspects of {A}ubry-{M}ather theory for quasiconvex
              {H}amiltonians},
   JOURNAL = {Calc. Var. Partial Differential Equations},
  FJOURNAL = {Calculus of Variations and Partial Differential Equations},
    VOLUME = {22},
      YEAR = {2005},
    NUMBER = {2},
     PAGES = {185--228},
      ISSN = {0944-2669},
   MRCLASS = {35F20 (37J50 49K10)},
  MRNUMBER = {2106767 (2006f:35023)},
MRREVIEWER = {Andrzej Swiech},
       DOI = {10.1007/s00526-004-0271-z},
       URL = {http://dx.doi.org/10.1007/s00526-004-0271-z},
}

@article {Is21,
    AUTHOR = {Ishii, Hitoshi},
     TITLE = {The vanishing discount problem for monotone systems of
              {H}amilton-{J}acobi equations. {P}art 1: linear coupling},
   JOURNAL = {Math. Eng.},
  FJOURNAL = {Mathematics in Engineering},
    VOLUME = {3},
      YEAR = {2021},
    NUMBER = {4},
     PAGES = {Paper No. 032, 21},
   MRCLASS = {35F21 (37N40 49L25 70H20)},
  MRNUMBER = {4147576},
MRREVIEWER = {Vincenzo Basco},
       DOI = {10.3934/mine.2021032},
       URL = {https://doi.org/10.3934/mine.2021032},
}

@article {IsLi20,
    AUTHOR = {Ishii, Hitoshi and Jin, Liang},
     TITLE = {The vanishing discount problem for monotone systems of
              {H}amilton-{J}acobi equations: part 2---nonlinear coupling},
   JOURNAL = {Calc. Var. Partial Differential Equations},
  FJOURNAL = {Calculus of Variations and Partial Differential Equations},
    VOLUME = {59},
      YEAR = {2020},
    NUMBER = {4},
     PAGES = {Paper No. 140, 28},
      ISSN = {0944-2669},
   MRCLASS = {35F21 (35B40 35F50 49L25)},
  MRNUMBER = {4135637},
MRREVIEWER = {Vincenzo Basco},
       DOI = {10.1007/s00526-020-01768-8},
       URL = {https://doi.org/10.1007/s00526-020-01768-8},
}

@article{IMTa17,
AUTHOR = {Ishii, Hitoshi and Mitake, Hiroyoshi and Tran, Hung V.},
     TITLE = {The vanishing discount problem and viscosity {M}ather
              measures. {P}art 1: {T}he problem on a torus},
   JOURNAL = {J. Math. Pures Appl. (9)},
  FJOURNAL = {Journal de Math\'ematiques Pures et Appliqu\'ees. Neuvi\`eme S\'erie},
    VOLUME = {108},
      YEAR = {2017},
    NUMBER = {2},
     PAGES = {125--149},
      ISSN = {0021-7824},
   MRCLASS = {35J60 (35B40 35J70 49L25)},
  MRNUMBER = {3670619},
       URL = {https://doi.org/10.1016/j.matpur.2016.10.013},
}

@article{IMTb17,
AUTHOR = {Ishii, Hitoshi and Mitake, Hiroyoshi and Tran, Hung V.},
     TITLE = {The vanishing discount problem and viscosity {M}ather
              measures. {P}art 2: {B}oundary value problems},
   JOURNAL = {J. Math. Pures Appl. (9)},
  FJOURNAL = {Journal de Math\'ematiques Pures et Appliqu\'ees. Neuvi\`eme S\'erie},
    VOLUME = {108},
      YEAR = {2017},
    NUMBER = {3},
     PAGES = {261--305},
      ISSN = {0021-7824},
   MRCLASS = {35J60 (35B40 35J70 49L25)},
  MRNUMBER = {3682741},
       URL = {https://doi.org/10.1016/j.matpur.2016.11.002},
}

@article {IsSic20,
    AUTHOR = {Ishii, Hitoshi and Siconolfi, Antonio},
     TITLE = {The vanishing discount problem for {H}amilton-{J}acobi
              equations in the {E}uclidean space},
   JOURNAL = {Comm. Partial Differential Equations},
  FJOURNAL = {Communications in Partial Differential Equations},
    VOLUME = {45},
      YEAR = {2020},
    NUMBER = {6},
     PAGES = {525--560},
      ISSN = {0360-5302},
   MRCLASS = {35F21 (35B40 37J99 49L25)},
  MRNUMBER = {4106998},
       DOI = {10.1080/03605302.2019.1710845},
       URL = {https://doi.org/10.1080/03605302.2019.1710845},
}

@article {IS11,
    AUTHOR = {Iturriaga, Renato and S{\'a}nchez-Morgado, H{\'e}ctor},
     TITLE = {Limit of the infinite horizon discounted {H}amilton-{J}acobi
              equation},
   JOURNAL = {Discrete Contin. Dyn. Syst. Ser. B},
  FJOURNAL = {Discrete and Continuous Dynamical Systems. Series B. A Journal
              Bridging Mathematics and Sciences},
    VOLUME = {15},
      YEAR = {2011},
    NUMBER = {3},
     PAGES = {623--635},
      ISSN = {1531-3492},
   MRCLASS = {49L25 (35F21 37J50)},
  MRNUMBER = {2774130 (2012c:49070)},
MRREVIEWER = {Diogo Lu{\'{\i}}s Aguiar Gomes},
       DOI = {10.3934/dcdsb.2011.15.623},
       URL = {http://dx.doi.org/10.3934/dcdsb.2011.15.623},
}

@unpublished{LPV,
Author = {Lions, P.-L. and Papanicolaou, G. and Varadhan, S.R.S.},
title = {Homogenization of {H}amilton-{J}acobi equation},
year = {1987 },
note={unpublished preprint}
}

@article {MiTra17,
    AUTHOR = {Mitake, Hiroyoshi and Tran, Hung V.},
     TITLE = {Selection problems for a discount degenerate viscous
              {H}amilton-{J}acobi equation},
   JOURNAL = {Adv. Math.},
  FJOURNAL = {Advances in Mathematics},
    VOLUME = {306},
      YEAR = {2017},
     PAGES = {684--703},
      ISSN = {0001-8708},
   MRCLASS = {35J62 (35B25 35J70 37J50 49L25)},
  MRNUMBER = {3581314},
MRREVIEWER = {Georgios Psaradakis},
       DOI = {10.1016/j.aim.2016.10.032},
       URL = {https://doi.org/10.1016/j.aim.2016.10.032},
}

@article {Z20,
    AUTHOR = {Zavidovique, Maxime},
     TITLE = {Convergence of solutions for some degenerate discounted
              {H}amilton-{J}acobi equations},
   JOURNAL = {Anal. PDE},
  FJOURNAL = {Analysis \& PDE},
    VOLUME = {15},
      YEAR = {2022},
    NUMBER = {5},
     PAGES = {1287--1311},
      ISSN = {2157-5045},
   MRCLASS = {35D40 (35F21 49L25)},
  MRNUMBER = {4492873},
       DOI = {10.2140/apde.2022.15.1287},
       URL = {https://doi.org/10.2140/apde.2022.15.1287},
}

@article{CFZZ24,
    AUTHOR = {Chen, Qinbo and Fathi, Albert and Zavidovique, Maxime and
              Zhang, Jianlu},
     TITLE = {Convergence of the solutions of the nonlinear discounted
              {H}amilton--{J}acobi equation: The central role of
              {M}ather measures},
   JOURNAL = {J. Math. Pures Appl. (9)},
    VOLUME = {181},
      YEAR = {2024},
     PAGES = {22--57},
       DOI = {10.1016/j.matpur.2023.10.006},
       URL = {https://doi.org/10.1016/j.matpur.2023.10.006},
}

@article{TZ26,
    AUTHOR = {Tu, Son N. T. and Zhang, Jianlu},
     TITLE = {Vanishing discount limits for fully nonlinear
              {H}amilton--{J}acobi equations on noncompact domains},
   JOURNAL = {Adv. Contin. Discrete Models},
    VOLUME = {2026},
      YEAR = {2026},
     PAGES = {Paper No.~58},
       DOI = {10.1186/s13662-026-04097-w},
       URL = {https://doi.org/10.1186/s13662-026-04097-w},
}

@article{GMT18,
    AUTHOR = {Gomes, Diogo A. and Mitake, Hiroyoshi and Tran, Hung V.},
     TITLE = {The selection problem for discounted {H}amilton--{J}acobi
              equations: some non-convex cases},
   JOURNAL = {J. Math. Soc. Japan},
    VOLUME = {70},
      YEAR = {2018},
    NUMBER = {1},
     PAGES = {345--364},
       DOI = {10.2969/jmsj/07017534},
       URL = {https://doi.org/10.2969/jmsj/07017534},
}

@unpublished{DNYZ24,
    AUTHOR = {Davini, Andrea and Ni, Panrui and Yan, Jun and
              Zavidovique, Maxime},
     TITLE = {Convergence/divergence phenomena in the vanishing discount
              limit of {H}amilton--{J}acobi equations},
      YEAR = {2024},
      NOTE = {Preprint, arXiv:2411.13780},
    EPRINT = {2411.13780},
ARCHIVEPREFIX = {arXiv},
PRIMARYCLASS = {math.AP},
       URL = {https://arxiv.org/abs/2411.13780},
}

@article{DI27,
    AUTHOR = {Davini, Andrea and Ishii, Hitoshi},
     TITLE = {The vanishing discount problem for nonlocal
              {H}amilton--{J}acobi equations},
   JOURNAL = {Nonlinear Anal.},
    VOLUME = {274},
      YEAR = {2027},
     PAGES = {Paper No.~114233},
       DOI = {10.1016/j.na.2026.114233},
       URL = {https://doi.org/10.1016/j.na.2026.114233},
}
\end{document}